\documentclass[11pt]{amsart}
\usepackage{amsfonts, amssymb, amsmath, mathrsfs}
\usepackage{tikz-cd}
\usepackage[hidelinks]{hyperref}
\usepackage{url}
\usepackage[left=25mm,right=25mm,top=38mm,bottom=34mm]{geometry}
\tikzset{
	labl/.style={anchor=south, rotate=90, inner sep=.5mm}
}
\newcommand{\spa}{\medskip}
\usepackage[utf8]{inputenc}

\usepackage{bbm}

\newcommand{\ag}{{\mbox{\larger[-4]$\dag$}}}

\newcommand{\PP}{\mathbb{P}}

\newcommand{\Q}{\mathbb{Q}}

\newcommand{\F}{\mathbb{F}}

\newcommand{\Z}{\mathbb{Z}}

\newcommand{\Crys}{\mathbf{Crys}}

\newcommand{\bt}{Barsotti--Tate\ }

\newcommand{\LS}{\mathbf{LS}}
\newcommand{\Foi}{\Fisoc^\mathrm{\ag}}
\newcommand{\oi}{\Isoc^\mathrm{\ag}}

\newcommand{\alg}{\mathrm{alg}}
\newcommand{\ka}{k^\mathrm{alg}}

\newcommand{\rev}{\mathrm{rev}}
\newcommand{\iso}{\xrightarrow{\sim}}

\newcommand{\Zp}{\mathbb{Z}_p}

\newcommand{\Qp}{\mathbb{Q}_p}
\newcommand{\Qpu}{{\mathbb{Q}_p^\mathrm{ur}}}

\newcommand{\Qpbar}{\overline{\mathbb{Q}}_{p}}

\newcommand{\Gm}{\mathbb{G}_m}

\newcommand{\tors}{\mathrm{tors}}
\newcommand{\crys}{\mathrm{crys}}
\newcommand{\ur}{\mathrm{ur}}

\newcommand{\sep}{\mathrm{sep}}

\newcommand{\Fp}{\F_p}

\newcommand{\Fpn}{\F_{p^n}}

\newcommand{\calA}{\mathcal{A}}

\newcommand{\calE}{\mathcal{E}}
\newcommand{\calEd}{{\calE}^\ag}

\newcommand{\calO}{\mathcal{O}}

\newcommand{\caloL}{\overline{\calL}}
\newcommand{\calL}{\mathcal{L}}
\newcommand{\calLd}{{\mathcal{L}^{\ag}}}
\newcommand{\calM}{\mathcal{M}}
\newcommand{\calMd}{{\calM^{\ag}}}

\newcommand{\calN}{\mathcal{N}}
\newcommand{\calNd}{{{\calN}^\ag}}

\newcommand{\caloN}{{\overline{\mathcal{N}}}}
\newcommand{\caloNd}{\caloN^\ag}
\newcommand{\calT}{\mathcal{T}}

\newcommand{\calV}{\mathcal{V}}

\newcommand{\VVec}{\mathbf{Vec}}

\newcommand{\Gr}{\mathrm{Gr}}
\newcommand{\Spec}{\mathrm{Spec}}

\newcommand{\Hom}{\mathrm{Hom}}

\newcommand{\End}{\mathrm{End}}

\newcommand{\GL}{\mathrm{GL}}
\newcommand{\id}{\mathrm{id}}

\newcommand{\PBQ}{\mathrm{PBQ}}
\newcommand{\PBS}{$\mathrm{PBS}$}
\newcommand{\MS}{\mathrm{MS}}

\newcommand{\cst}{\mathrm{cst}}
\newcommand{\rk}{\mathrm{rk}}
\newcommand{\et}{{\mathrm{\acute{e}t}}}

\newcommand{\Fisoc}{\mathbf{F\textrm{-}Isoc}}
\newcommand{\Fiisoc}{\mf F^{\infty}\textrm{-} \mf{Isoc}}
\newcommand{\Fiisoce}{\mf F^{\infty}\textrm{-} \mf{Isoc}^\et}
\newcommand{\Fnisoc}{\mf F^{n}\textrm{-} \mf{Isoc}}
\newcommand{\Fnoi}{\mf F^{n}\textrm{-} \mf{Isoc}^\ag}
\newcommand{\Fioi}{\mf F^{\infty}\textrm{-} \mf{Isoc}^\ag}
\newcommand{\oil}{\mf{Isoc}^\ag_\Qpi}
\newcommand{\Isoc}{\mathbf{Isoc}}
\newcommand{\Isocl}{\mathbf{Isoc}_{\Qpi}}
\newcommand{\Isocle}{\mathbf{Isoc}_{\Qpi}^\et}

\newcommand{\DM}{Dieudonné--Manin}
\newcommand{\Qpi}{{\Q_{p}^\ur}}

\newcommand{\Qpn}{{\Q_{p^n}}}

\newcommand{\Qpis}{$\Qpi$-structure}
\newcommand{\DMs}{\DM\ \Qpis}

\newcommand{\dhl}{$\dag$-hull}
\newcommand{\dex}{$\dag$-ex\-ten\-da\-ble}

\newcommand{\dc}{$\dag$-com\-pac\-ti\-fi\-cation}
\newcommand{\dcs}{$\dag$-com\-pac\-ti\-fi\-cations}

\newcommand{\tcalE}{\calE_{\alg}}
\newcommand{\tcalEd}{\calE_{\alg}^\ag}

\newcommand{\tMd}{M_{\alg}^\ag}

\newcommand{\pie}{\pi_1^{\et}}
\newcommand{\mf}{\mathbf}

\begin{document}

	\newtheorem{theo}[subsubsection]{Theorem}
	\newtheorem*{theo*}{Theorem}
	\newtheorem{ques}[subsubsection]{Question}
	\newtheorem*{ques*}{Question}
	\newtheorem{conj}[subsubsection]{Conjecture}
	\newtheorem{prop}[subsubsection]{Proposition}
	\newtheorem{lemm}[subsubsection]{Lemma}
	\newtheorem*{lemm*}{Lemma}
	\newtheorem{coro}[subsubsection]{Corollary}
	\newtheorem*{coro*}{Corollary}

	\theoremstyle{definition}
	\newtheorem{defi}[subsubsection]{Definition}
	\newtheorem*{defi*}{Definition}
	\newtheorem{hypo}[subsubsection]{Hypothesis}
	\newtheorem{rema}[subsubsection]{Remark}
	\newtheorem{exam}[subsubsection]{Example}
	\newtheorem{nota}[subsubsection]{Notation}
	\newtheorem{cons}[subsubsection]{Construction}
	
	\numberwithin{equation}{subsection}
	
	\title{Parabolicity conjecture of $F$-isocrystals}
	\date{\today}
	\makeatletter
			\@namedef{subjclassname@2020}{%
		\textup{2020} Mathematics Subject Classification}
	\makeatother
	
	
	
	\author{Marco D'Addezio}
\address{Institut de Mathématiques de Jussieu-Paris Rive Gauche, SU - 4 place Jussieu, Case 247, 75005 Paris}
\email{daddezio@imj-prg.fr}

	\begin{abstract}
		In this article we prove Crew's parabolicity conjecture of $F$-isocrystals. For this purpose, we introduce and study the notion of $\dagger$-hull of a sub-$F$-isocrystal. On the way, we prove a new Lefschetz theorem for overconvergent $F$-isocrystals.
	\end{abstract}
	
	\maketitle

	\tableofcontents
	\section{Introduction}
	\subsection{Main results}
Let $X$ be a smooth geometrically connected variety over a perfect field $k$ of positive characteristic $p$ and let $n$ be a positive integer. For an overconvergent $F^n$-isocrystal $(\calMd,\Phi_\calM^\ag)$ over $X$ we write $(\calM,\Phi_\calM)$ for the associated $F^n$-isocrystal and we suppose that $(\calM,\Phi_\calM)$ admits the slope filtration $$0=S_0(\calM)\subsetneq S_1(\calM) \subsetneq...\subsetneq S_m(\calM)=\calM.$$ If $\eta$ is a point of $X$ with perfect residue field, Crew defined in \cite{Cre92a} two algebraic \textit{monodromy groups} for $\calMd$ with respect to $\eta$, that we denote by $G(\calM,\eta)$ and $G(\calMd,\eta)$. The former algebraic group is a subgroup of the latter and they both are subgroups of $\GL(\calM_\eta)$, where $\calM_\eta$ is the fibre of $\calM$ at $\eta$. In this article we prove the following fundamental result about these groups.

\begin{theo}[Theorem \ref{para-c:t}]\label{i-para-conj:t}
The subgroup $G(\calM,\eta)\subseteq G(\calMd,\eta)$ is the subgroup of $G(\calMd,\eta)$ stabilising the slope filtration of $\calM_\eta$. Moreover, when $\calMd$ is semi-simple, $G(\calM,\eta)$ is a parabolic subgroup of $G(\calMd,\eta)$.
\end{theo}
This solves the \textit{parabolicity conjecture}, initially proposed as a question in \cite[page 460]{Cre92a}. Partial results on this conjecture have been obtained in \cite{Cre92b}, \cite{Crew94}, \cite{Tsu02}, \cite{AD18}, and \cite{Tsu19}. Theorem \ref{i-para-conj:t} can be seen as a natural enhancement of Kedlaya's full faithfulness theorem, proven in \cite{KedFull}.

\spa


When $k$ is a finite field and $\calMd$ is semi-simple, the group $G(\calMd,\eta)$ is “the same” as the geometric monodromy group of the semi-simple $\ell$-adic \textit{companions} of $(\calMd,\Phi_\calM^\ag)$, namely those semi-simple $\ell$-adic lisse sheaves with the same $L$-function as $(\calMd,\Phi_\calM^\ag)$ (see \cite{Pal15}, \cite{Dri18}, and \cite{Dad}). In this particular case, we deduce from Theorem \ref{i-para-conj:t} the following semi-simplicity result.
\begin{theo}[Theorem \ref{semi-simp:t}]\label{i-semi-simp:t}
Let $X$ be a smooth variety over a finite field and $f:A\to X$ an abelian scheme with constant slopes. If $(\calM,\Phi_\calM)$ is the $F$-isocrystal $R^1f_{\crys*}\calO_{A,\crys}$, the induced $F$-isocrystal $\Gr_{S_\bullet}(\calM,\Phi_\calM)$ is semi-simple. In particular, $R^1f_{\et*}\underline{\Qp}$ is a semi-simple lisse $\Qp$-sheaf over $X$.
\end{theo}

As an additional outcome of the theory developed, we are also able to prove a new diophantine result for abelian varieties.

\begin{theo}[Theorem \ref{abelian:t}]\label{intro-abelian:t}
	Let $E$ be a finitely generated field extension of $\Fp$ and let $A$ be an abelian variety over $E$. The group $A(E^{\mathrm{sep}})[p^\infty]$ is finite in the following two case.
	\begin{itemize}
		\item[(i)] If $\End(A)\otimes_\Z \Qp$ is a division algebra.
		\item[(ii)] If $\End(A)\otimes_\Z \Q$ has no factor of Albert-type IV.
	\end{itemize}
\end{theo}
This theorem enriches the list of known results on the finiteness of separable $p$-torsion points of abelian varieties (see \cite{Vol95} and \cite{Ros17}). It is worth mentioning that abelian varieties with finite separable $p$-torsion play an important role in the theory of Brauer--Manin obstructions in positive characteristic (see \cite{PV10}).

\spa

As a further consequence of Theorem \ref{i-para-conj:t}, we solve a conjecture proposed by Kedlaya in \cite[Rmk. 5.14]{Ked16}.

\begin{coro}[Corollary \ref{k-c:c}]\label{i-k-c:c}
Let $X$ be a smooth connected variety over a perfect field $k$ and let $(\calM_1^\ag,\Phi_{\calM_1}^\ag)$ and $(\calM_2^\ag,\Phi_{\calM_2}^\ag)$ be two irreducible overconvergent $F^n$-isocrystals over $X$ with constant slopes. If $(S_1(\calM_1),\Phi_{\calM_1}|_{S_1(\calM_1)})$ and $(S_1(\calM_2),\Phi_{\calM_2}|_{S_1(\calM_2)})$ are isomorphic $F^n$-isocrystals, then $(\calM_1^\ag,\Phi_{\calM_1}^\ag)$ and $(\calM_2^\ag,\Phi_{\calM_2}^\ag)$ are isomorphic overconvergent $F^n$-isocrystals.
\end{coro}

When $X$ has dimension $1$ or $k$ is a finite field the conjecture is proved in \cite{Tsu19}. Combining Corollary \ref{i-k-c:c} over finite fields, the Langlands reciprocity conjecture for overconvergent isocrystals, proven in \cite{Abe}, and Chebotarev density theorem, we also get as a consequence the following stronger form of the multiplicity one theorem for cuspidal automorphic representations.

\begin{theo}[Theorem \ref{mult-one:t}]\label{i-mult-one:t}
	Let $E$ be a global field of characteristic $p$ and let $\mathbb{A}$ be its adele ring. For every positive integer $r$, the isomorphism class of a $\Qpbar$-linear cuspidal automorphic representation of $\GL_r(\mathbb{A})$ is determined by the datum of the Hecke eigenvalues of minimal slope at all but finitely many places of $E$.
\end{theo}

\subsection{\texorpdfstring{$\dag$}{}-hull of \texorpdfstring{$F$}{}-isocrystals}

One of the main tools introduced and studied in this article is the notion of \textit{$\dag$-hull} of a sub-$F^n$-isocrystal.
 \begin{defi}
	Let $(\calN,\Phi_\calN)\subseteq (\calM,\Phi_\calM)$ be an inclusion of $F^n$-isocrystals. The \textit{\dhl\ }of $(\calN,\Phi_\calN)$ in $(\calM,\Phi_\calM)$ is the smallest subobject of $(\calM,\Phi_\calM)$ containing $(\calN,\Phi_\calN)$ and coming from an overconvergent $F^n$-isocrystal. We denote it by $(\caloN,\Phi_\caloN)$. 
\end{defi}

This notion was introduced by the author in September 2018, during a discussion in Berlin with Abe and Esnault. In that case, it was used as an attempt to better understand the results in \cite{AD18}. A couple of months later, reading \cite{Vel91}, the author understood that a certain property on the slopes of the $\dag$-hull of an $F^n$-isocrystal would have implied the parabolicity conjecture. This property was satisfied, for example, by certain $F^n$-isocrystals coming from $p$-divisible groups, as a consequence of a local result proven by Tate \cite[Prop. 12]{Tat67}. In this article, we prove that the expected property for the $\dag$-hull is satisfied in general.

\begin{theo}[Theorem \ref{main:t}]\label{i-main:t}
	Let $X$ be a smooth variety over a perfect field $k$ and let $(\calN,\Phi_\calN)\subseteq (\calM,\Phi_\calM)$ be an inclusion of $F^n$-isocrystals over $X$. If $\calM$ comes from an overconvergent isocrystal and has constant slopes, then $S_1(\calN,\Phi_\calN)=S_1(\caloN,\Phi_{\caloN})$.
\end{theo}

When $X$ is a curve, we first reduce to the case $X=\mathbb{A}^1_k$, then we prove an analogous result on the generic point of $\mathbb{A}^1_k$. Finally, we deduce the global result from the generic one. In dimension 1, Theorem \ref{i-main:t} is essentially proven in \cite[Prop. 6.1]{Tsu19} as well. In his proof Tsuzuki makes use of a certain filtration for overconvergent $F$-isocrystals, namely the \textit{$\PBQ$ filtration} that he constructs in \cite[Thm. 3.27]{Tsu19}. In the proof we propose in this article, instead, we avoid [\textit{ibid.}, Thm. 3.27]. On the other hand, we use Theorem \ref{i-main:t} to recover [\textit{ibid.}, Thm. 3.27] and extend it to arbitrary smooth varieties (Corollary \ref{PBS:c}).

\spa

For higher dimensional varieties we deduce Theorem \ref{i-main:t} from the case of curves, thanks to a new Lefschetz theorem for overconvergent $F^n$-isocrystals (Theorem \ref{i-tame-lef:t}). We postpone for a moment the discussion about this part. Before that, let us briefly explain how to deduce Theorem \ref{i-para-conj:t} from Theorem \ref{i-main:t}. We prove first, thanks to Chevalley's theorem, that Theorem \ref{i-main:t} implies a variant of Theorem \ref{i-para-conj:t} for some slightly different groups, namely the monodromy groups of \textit{$F^\infty$-isocrystals} (Proposition \ref{para:p}). Subsequently, to pass from this variant to the original statement, we introduce a third type of monodromy groups, namely the monodromy groups of \textit{isocrystals with punctual $\Qpu$-structure}, defined in §\ref{Fi-mono:d}. In the cases we consider, these latter monodromy groups are $\Qpu$-forms of Crew's monodromy groups (the ones defined in \cite{Cre92a}), thus we are then able to prove Theorem \ref{i-para-conj:t}. Let us see now more in details the technical points we mentioned.

\subsection{Lefschetz theorem}
The main issue to reduce Theorem \ref{i-main:t} to the case of curves is due to the existence of wild ramification in positive characteristic. One would like to find a smooth connected curve $C\subseteq X$ such that for every overconvergent isocrystal $\calMd$ over $X$, the Tannakian category $\langle \calM^\ag\rangle$ spanned by $\calMd$ (see §\ref{tann:ss}) is equivalent to the Tannakian category $\langle \calM^\ag|_C\rangle$ spanned by the restriction of $\calMd$ to $C$. This is possible, for example, for local systems in characteristic $0$, or for tamely ramified $\ell$-adic lisse sheaves in positive characteristic (see \cite{Esn}). The failure of the existence of such a nice curve for general $\ell$-adic lisse sheaves is already clear for $\mathbb{A}^2_k$ (see [\textit{ibid.}, Lem. 5.4]). On the other hand, if rather than considering all the objects at the same time one focuses on one object at a time, then such a nice curve exists over finite fields both for $\ell$-adic lisse sheaves and overconvergent $F^n$-isocrystals (see \cite[Lem. 6 and Thm. 8]{Kat99} and \cite[Thm. 3.10]{AE19}). We extend this result to \textit{docile} overconvergent $F^n$-isocrystals over general perfect fields, namely those overconvergent $F^n$-isocrystals which admit a log-extension with nilpotent residues.

 \begin{theo}[Theorem \ref{tame-lef:t}]\label{i-tame-lef:t}
	Let $Y\subseteq \PP^d_{\ka}$ be a smooth connected projective variety of dimension at least $2$ and let $D$ be a simple normal crossing divisor. If $(\calMd, \Phi_\calM^\ag)$ is an overconvergent $F^n$-isocrystal over $X:=Y\setminus D$ docile along $D$, then there exists a smooth connected curve $C\subseteq X$ such that the restriction functor $\langle \calMd\rangle \to \langle \calMd|_{C} \rangle$
	is an equivalence of categories.
\end{theo}

This theorem is obtained by a combination of various Lefschetz-type results. One of the main ingredients is Theorem \ref{AE-lef:t}, proven by Abe--Esnault, which gives a class of curves $C$ such that the restriction functor $\langle \calMd\rangle \to \langle \calMd|_{C} \rangle$ is fully faithful. To prove Theorem \ref{i-tame-lef:t}, we show that for at least one of these curves the restriction functor is also essentially surjective. This condition can be tested on rank $1$ objects, which have the advantage to come from $p$-adic characters of the étale fundamental group and they are easier to extend from $C$ to $X$. The difficult part is to impose that the extended characters come from overconvergent $F^n$-isocrystals. In our proof, this is done by combining Lemma \ref{ramif:l} and Proposition \ref{rk1-unr:p}. Interestingly, in order to use Proposition \ref{rk1-unr:p} we need Theorem \ref{i-main:t} for curves.  Therefore, the proofs of Theorem \ref{i-main:t} and Theorem \ref{i-tame-lef:t} are intrinsically intertwined.

\subsection{Punctual \texorpdfstring{$\Qpi$}{}-structures} The other technical issue we have to solve in our article is given by the fact that the field of scalars of the Tannakian category of isocrystals is in general bigger than the field of scalars of the category of $F^n$-isocrystals. We encounter this issue in the proof of Theorem \ref{i-para-conj:t}, which is easier for monodromy groups of objects endowed with an $F^n$-structure (knowing Theorem \ref{main:t}), but it is harder for the monodromy groups of the objects without $F^n$-structure. To jump from one setting to the other we slightly modify both categories. We first replace the category of $F^n$-isocrystals with the category of $F^\infty$-isocrystals, namely the 2-colimit of the categories of $F^n$-isocrystals for various $n$. If $k$ is big enough, this new category is a $\Qpu$-linear category. On the other side, we construct the category of isocrystals with punctual $\Qpi$-structure (see §\ref{isoc-p:d}), which is simply the category of isocrystals endowed with the choice of a $\Qpi$-linear lattice at some fibre. This other Tannakian category is also $\Qpi$-linear. 

\spa

There is a natural functor between these two categories obtained thanks to the observation that a finite-rank $F^\infty$-isocrystal over the spectrum of an algebraically closed field has a natural $\Qpi$-linear lattice (Lemma \ref{DM:l}). To prove Theorem \ref{i-para-conj:t} we relate the monodromy groups of these objects and their overconvergent variants thanks to Proposition \ref{fund-exac-seq:p}, which is an analogue of the
homotopy exact sequence for the étale fundamental group.

\subsection{The structure of the article}

In §\ref{nota:s} we fix most of the notation we use in our article. In §\ref{isoc-with-qpu:s} we introduce and study isocrystals with punctual $\Qpu$-structure and we prove Proposition \ref{fund-exac-seq:p}. In §\ref{main-theo:s} we prove all the main theorems of our article. In particular, in §\ref{curve-case:ss} we prove Theorem \ref{i-main:t} over curves, in §\ref{Chevalley:ss} we prove that Theorem \ref{i-main:t} implies a variant of Theorem \ref{i-para-conj:t}, and in §\ref{Lefschet:ss} we prove the Lefschetz theorem (Theorem \ref{i-tame-lef:t}) and we use it to complete the proof of Theorem \ref{i-main:t} in §\ref{pf-main-t:sss}, which is then used to prove Theorem \ref{i-para-conj:t}. In §\ref{apps:s} we deduce some consequences of the results in §\ref{main-theo:s}. More precisely, we first prove Theorem \ref{i-semi-simp:t} in §\ref{fini-fiel:ss}. Subsequently in §\ref{sep-pts:ss} we prove Theorem \ref{intro-abelian:t} and in §\ref{k-c:ss} we prove Kedlaya's conjecture (Corollary \ref{i-k-c:c}) and its consequence, namely Theorem \ref{i-mult-one:t}. Finally, in §\ref{PBS:ss}, we prove the existence of Tsuzuki's PBQ filtration for general smooth varieties (Corollary \ref{PBS:c}) and we prove a generalisation of Kedlaya's conjecture (Corollary \ref{gen-Ked-con:c}).

\subsection{Acknowledgements}I am grateful to João Pedro dos Santos for pointing out the characterisation of parabolic subgroups in \cite{Vel91}. This strongly motivated my study of the $\dag$-hulls of sub-$F^n$-isocrystals. It is also a pleasure to thank Tomoyuki Abe, Emiliano Ambrosi, Anna Cadoret, Hélène Esnault, Quentin Guignard, Kiran Kedlaya, Damian R\"ossler, Peter Scholze, Atsushi Shiho, and Nobuo Tsuzuki for several helpful discussions. In particular, I thank Kedlaya for explaining to me the proof of Proposition \ref{local:p}. Finally, I thank the referees for many enlightening comments which have greatly helped to correct and improve the article.
	
	\spa
	
	 The author was funded by the Deutsche Forschungsgemeinschaft (DFG, German Research Foundation) under Germany's Excellence Strategy – The Berlin Mathematics
	Research Center MATH+ (EXC-2046/1, project ID: 390685689) and by the Max-Planck Institute for Mathematics.
 \section{Notation}\label{nota:s}
 \subsection{Tannakian categories}\label{tann:ss} If $\mathbf{C}$ is a Tannakian category and $A$ is an object of $\mathbf{C}$, we write $H^0(A)$ for the vector space $\Hom(\mathbbm{1},A)$ and $\langle A \rangle$ for the Tannakian subcategory\footnote{We will always require that a Tannakian subcategory is closed under the operation of taking subobjects.} of $\mathbf{C}$ spanned by $A$. The \textit{socle} of $A$ is defined to be the maximal semi-simple subobject of $A$ and the \textit{socle filtration} of $A$ is the unique ascending filtration $0=A_0\subsetneq A_1\subsetneq...\subsetneq A_n=A$ such that $A_i/A_{i-1}$ is the socle of $A/A_{i-1}$ for every $i$. 
 \spa 
 
 Throughout the article we will freely make use of the notion of \textit{observable functor}, as presented in \cite[§A]{DE20}, and all the criteria listed there to determine various properties of morphisms of Tannaka groups (e.g. when a morphism is a closed immersion, faithfully flat, etc.). We recall that by [\textit{ibid.}, Prop. A.3], an exact functor between Tannakian categories $\Psi:\mathbf{C}\to \mathbf{D}$ is observable if and only if every subquotient of an object of the essential image of $\Psi$ can be embedded into an object in the essential image of $\Psi$. We will also use the notion of \textit{scalar extensions} of $K$-linear Tannakian categories with respect to field extension $L/K$ (possibly infinite), as presented in \cite[§4.4]{Del89} and \cite[Thm. 3.1.3]{Sta08}.
 
\subsection{Isocrystals}
Let $p$ be a prime number. We fix an algebraic closure of $\Fp$, denoted by $\F$, and for $n\in \Z_{>0}$ we write $\Fpn$ for the subfield of $\F$ with $p^n$ elements. For a perfect field $k$ of positive characteristic $p$ we write $W(k)$ for its ring of Witt vectors, $K(k)$ for its fraction field, and $\sigma$ for the lift of the Frobenius of $k$. We also write $\Qpn$ for $K(\Fpn)$ and $\Qpi$ for the union $\cup_{n>0}\Qpn\subseteq K(\F)$. If $X$ is a smooth variety over $k$, we denote by $\Crys(X)$ the category of crystals of coherent $\calO_{X,\crys}$-modules over the absolute crystalline site of $X$ and by $\Isoc(X)$ its localisation $\Crys(X)[\tfrac 1p]$\footnote{To prove the results of our article we could have simply considered the smaller category of convergent isocrystals since we are interested in $F^n$-isocrystals, which are convergent by \cite[Thm. 2.2]{Ked16}}. We also denote by $\oi(X)$ the category of (coherent) overconvergent isocrystals over $X/K(k)$ (cf. \cite{Ber96}).

\spa

We say that $\eta$ is a \textit{perfect point} (resp. \textit{geometric point}) of $X$, if it is an $\Omega$-point for some field extension $\Omega/k$ with $\Omega$ perfect (resp. algebraically closed). We write $\omega_\eta$ for the exact $\otimes$-functor $\Isoc(X)\to \VVec_{K(\Omega)}$ which sends an isocrystal over $X$ to the pullback over $\Spec(\Omega)$. We denote with the same symbol the analogous exact $\otimes$-functor $\oi(X)\to \VVec_{K(\Omega)}$. If $X$ is connected, for every $\calM\in \Isoc(X)$ we denote by $G(\calM,\eta)$ the Tannaka group of $\langle \calM \rangle$ with respect to $\omega_\eta$. If $\calM$ is convergent, this group coincides with the group $\mathrm{DGal}(\calM,\eta)$ defined in \cite[§2]{Cre92a}. Similarly, for $\calM^\ag\in \Isoc^\ag(X)$ we denote by $G(\calM^\ag,\eta)$ the Tannaka group of $\langle \calM^\ag \rangle$ with respect to $\omega_\eta$. This latter group coincides with Crew's monodromy group $\mathrm{DGal}(\calM^\ag,\eta)$. 
 
 \spa
 
 We write $F:X\to X$ for the absolute Frobenius of $X$ and for a positive integer $n$, we write $\Fnisoc(X)$ (resp. $\Fnoi(X)$) for the category of pairs $(\calM,\Phi_\calM)$ (resp. $(\calMd,\Phi_\calM^\ag)$) where $\calM\in \Isoc(X)$ (resp. $\calMd\in \oi(X)$) and $\Phi_\calM$ is a $K$-linear isomorphism $(F^n)^*\calM\iso \calM$ (resp. $K$-linear isomorphism $(F^n)^*\calMd\iso \calMd$). The category $\Fnisoc(X)$ is canonically equivalent to the category of convergent $F^n$-isocrystals (see \cite[Thm. 2.2]{Ked16}). If $(\calM,\Phi_\calM)$ is an $F^n$-isocrystal we define inductively $\Phi_\calM^m:=\Phi_{\calM}\circ F^*(\Phi_\calM^{m-1})$ for $m>0$, where $\Phi_\calM^1:=\Phi_\calM$ and we say that these $F^{nm}$-structures are the powers of $\Phi_{\calM}$. We write $\Fiisoc(X)$ for $2\textrm{-}\varinjlim_n \Fnisoc(X)$ and $\Fioi(X)$ for $2\textrm{-}\varinjlim_n \Fnoi(X)$. If $(\calM,\Phi_\calM)$ is an $F^n$-isocrystal, we write $(\calM,\Phi_\calM^{\infty})$ for its image in $\Fiisoc(X)$ and we use the analogous convention in the overconvergent setting.
 
 \spa

  We write $\Isoc(X)_F$ (resp. $\oi(X)_F$) for the subcategory of $\Isoc(X)$ (resp. $\oi(X)_F$) spanned by those isocrystals (resp. overconvergent isocrystals) which admit some $F$-structure. The natural “restriction” functor $\alpha:\oi(X)_F\to \Isoc(X)_F$ is fully faithful (see \cite{Ked07} and \cite[Cor. 5.7]{DE20}). We say that an $\calM\in \Isoc(X)_F$ is \textit{$\dag$-extendable} if it is in the essential image of $\alpha$ and we write $\calMd$ for the associated overconvergent isocrystal (which is unique up to isomorphism).
 
 \spa
 
By \cite[Cor. 2.3.1]{Kat79}, \cite[Thm. 4.1]{dO99}, and \cite[Thm. 3.12]{Ked16}, after possibly removing a divisor of $X$, an $F^n$-isocrystal $(\calM,\Phi_\calM)$ has constant slopes. In this case, thanks to \cite[Cor. 2.6.2]{Kat79} and \cite[Cor. 4.2]{Ked16}, it acquires the slope filtration $$0=S_0(\calM)\subsetneq S_1(\calM) \subsetneq...\subsetneq S_m(\calM)=\calM$$ where each graded piece $S_i(\calM)/S_{i-1}(\calM)$ is of pure slope $s_i\in \Q$ and $s_1<s_2<...<s_m$.

 \section{Isocrystals with punctual \texorpdfstring{$\Qpi$}{}-structure} \label{isoc-with-qpu:s}
 
 \subsection{First definitions}  Let $k$ be a perfect field of characteristic $p$ and $X$ a smooth variety over $k$. If one wants to compare the Tannakian category spanned by an $F^n$-isocrystal $(\calM,\Phi_{\calM})$ over $X$ with the one spanned by the isocrystal $\calM$ in $\Isoc(X)$ one encounters issues related to the different fields of scalars of the categories. Indeed, while the former object lives in $\Fnisoc(X)$, which is a $\Qpn$-linear Tannakian category, the latter is in $\Isoc(X)$, a $K$-linear Tannakian category where $K:=K(k)$. We propose here a possible approach to compare the two categories. This will lead to Proposition \ref{fund-exac-seq:p}.
 
 \begin{hypo}\label{hypo}
Throughout §\ref{isoc-with-qpu:s}, we assume that $k$ is endowed with the choice of an inclusion $\F\subseteq k$.\footnote{One can easily generalise what we do in this section to the case when $k$ does not contain $\F$. We preferred to avoid this discussion here since the parabolicity conjecture, as we will see, is a “geometric statement”, meaning that we are allowed to replace the field $k$ with its algebraic closure.} In particular, for every perfect field extension $\Omega/k$ we have a preferred embedding $\Qpi\hookrightarrow K(\Omega)$.
 	\end{hypo}
 
 %
 
  \begin{defi}\label{isoc-p:d}
  If $\eta$ is a geometric point of $X$, an \textit{isocrystal with (punctual) $\Qpi$-structure} over $(X,\eta)$ is a pair $(\calM,V_\calM)$ where $\calM$ is an isocrystal and $V_\calM$ is a $\Qpi$-linear lattice of $\omega_\eta(\calM)$, namely a $\Qpi$-linear vector subspace $V_\calM\subseteq \omega_\eta(\calM)$ such that $V_\calM\otimes_{\Qpi}K(\Omega)= \omega_\eta(\calM)$. A morphism of isocrystals with $\Qpi$-structure $(\calM,V_\calM)\to(\calN,V_\calN)$ is a morphism of isocrystals $f:\calM\to \calN$ such that $\omega_\eta(f)(V_\calM)\subseteq V_\calN$. We write $\Isoc_{\Qpi}(X,\eta)$ for the category of isocrystals with punctual $\Qpi$-structure over $(X,\eta)$. 
 	
 \end{defi}

\begin{lemm}If $X$ is geometrically connected, the category $\Isocl(X,\eta)$ has a natural structure of a $\Qpi$-linear neutral Tannakian category.
\end{lemm}
\begin{proof} 
	The $\otimes$-structure is the one induced by the $\otimes$-structures of $\Isoc(X)$ and $\VVec_{\Qpi}$, with unit object $\mathbbm 1=(\calO_{X/K},\Qpi)$. Thanks to the assumption that $X$ is geometrically connected and $k$ contains $\F$, we deduce that $\End(\mathbbm 1)=\Qpi$. We claim that $\Isocl(X,\eta)$ is rigid. Indeed, for an isocrystal with $\Qpi$-structure $(\calM,V_\calM)$, we may take the $\Qpi$-linear subspace $V_{\calM^\vee}\subseteq \omega_x(\calM^\vee)$ corresponding to those morphisms $f\in \Hom(\omega_x(\calM),K)$ such that $f(V)\subseteq \Qpi$. The isocrystal with $\Qpi$-structure $(\calM,V_\calM)^\vee:= (\calM^\vee,V_{\calM^\vee})$, endowed with the natural morphisms $\mathrm{ev}: (\calM,V_\calM)\otimes (\calM,V_\calM)^\vee\to \mathbbm{1}$ and $\delta: \mathbbm{1} \to (\calM,V_\calM)^\vee\otimes (\calM,V_\calM)$, is then a dual of $(\calM,V_\calM)$. Finally, we note that $\omega_{\eta,\Qpi}:\Isocl(X,\eta)\to \VVec_{\Qpi}$ defined by $(\calM,V_\calM)\mapsto V_\calM$ is a fibre functor for $\Isocl(X,\eta)$. This ends the proof.\end{proof}
 
 Now that we have constructed a neutral Tannakian category of isocrystals with smaller field of scalars, we want to study its interaction with the category of $F^n$-isocrystals.

\begin{cons}\label{lattice:cons} Let $\Omega/k$ be an algebraically closed field extension. For an $F^n$-isocrystal $(M,\Phi_M)$ over $\Omega$ (of finite-rank), we write $\omega_\Qpi(M,\Phi_M)\subseteq M$ for the $\Qpi$-linear vector subspace of vectors $v\in M$ such that $\Phi_M^iv=p^jv$ for some $(i,j)\in \Z_{>0}\times \Z$.
\end{cons}

\begin{lemm}[after Dieudonné--Manin]\label{DM:l}
The vector space $\omega_\Qpi(M,\Phi_M)$ is a $\Qpi$-linear lattice of the $K(\Omega)$-vector space $M$.
	\end{lemm}
	\begin{proof}
Thanks to Dieudonné--Manin classification, after possibly taking powers of $\Phi_M$, we may assume that $(M,\Phi_M)=(K(\Omega),p^s\sigma)$ for some $s\in \Z$. The result then follows from the observation that $\Qpi=\cup_nK(\Omega)^{\sigma^n=\id}$.
	\end{proof}
\begin{defi}\label{Fi-lattices:d}
	 We say that $\omega_\Qpi(M,\Phi_M)$ is the \textit{Dieudonné--Manin $\Qpi$-structure} of $(M,\Phi_M)$. The assignment $(M,\Phi_M)\mapsto \omega_\Qpi(M,\Phi_M)$ produces a $\Qpi$-linear fibre functor $\omega_\Qpi:\Fiisoc(\Omega)\to \VVec_{\Qpi}$. If $\eta$ is a $\Omega$-point of $X$, we write $$\omega_{\eta,\Qpi}:\Fiisoc(X)\to \VVec_{\Qpi}$$ for the composition $\omega_\Qpi\circ \eta^*$ and we write $$\Lambda_\eta: \Fiisoc(X)\to \Isocl(X,\eta)$$ for the functor obtained by sending $(\calM,\Phi_\calM^\infty)\mapsto(\calM,\omega_{\eta,\Qpi}(\calM,\Phi_\calM^\infty))$. We say that an object in the essential image of $\Lambda_\eta$ is an \textit{isocrystal with Dieudonné--Manin $\Qpi$-structure} over $(X,\eta)$.
\end{defi}

\begin{rema}\label{DM:r}
The existence of the \DMs\ of an $F$-isocrystal over an algebraically closed field has its own interest. For example, thanks to \cite{Ked06}, if $k$ is any field of characteristic $p$, one can associate to a variety $X/k$ the finite-dimensional $\Qpi$-linear vector spaces $\omega_{\Qpi}(H^i_{\rm{rig}}(X_{k^\alg}/K(k^{\alg})))$ and their variant with compact support. This assignment produces a $\Qpi$-linear cohomology theory with all the desired properties (e.g. Poincaré duality, the K\"unneth formula, etc.). This solves in a minimal way Serre's obstruction to the existence of a $\Qp$-linear cohomology theory.
\end{rema}


\subsection{Basic properties}

In general, the category $\langle\calM,V_\calM\rangle$ spanned by an isocrystal with $\Qpi$-structure might by quite different from the category spanned by the associated isocrystal $\calM$ in $\Isoc(X)$. For example, $(\calM,V_\calM)$ might be irreducible even if $\calM$ is not, or it might happen that $H^0(\calM,V_\calM)$ has lower dimension than $H^0(\calM)$. In what follows we show that these phenomena do not occur for \DMs s.

\begin{lemm}\label{H0:l}If $(\calM,\Phi_{\calM})$ is an $F^n$-isocrystal, the maximal trivial\footnote{ We say that an object in a Tannakian category is trivial if it is isomorphic to a direct sum $\mathbbm{1}^{\oplus m}$ for some $m\geq0$. } subobject $\calT\subseteq\calM$ is kept stable under $\Phi_\calM$. In particular, if $(\calM,V_\calM)$ is the induced isocrystal with Dieudonné--Manin $\Qpi$-structure then $$H^0(\calM,V_\calM)\otimes_{\Qpi}K\simeq H^0(\calM).$$ 
\end{lemm}

\begin{proof}
By maximality $\calT\subseteq \calM$ is preserved by $\Phi_\calM$, so that it defines an inclusion $(\calT,V_\calM|_\calT)\subseteq (\calM,V_\calM)$ of isocrystals with $\Qpi$-structure. Since $\calT$ is trivial, $(\calT,V_\calM|_\calT)$ is a trivial isocrystal with $\Qpi$-structure. This yields the desired result.
\end{proof}

\begin{lemm}\label{ss-uniq-DM:l}
	A semi-simple isocrystal $\calM$ admits at most one \DM\ $\Qpi$-structure up to isomorphism. In addition, if $V_\calM$ is such a $\Qpi$-structure, $(\calM,V_\calM)$ decomposes as the direct sum $\bigoplus_i (\calM_i,V_{\calM_i})^{\oplus a_i}$ with $a_i>0$ and the isocrystals $\calM_i$ irreducible and pairwise non-isomorphic.
\end{lemm}
\begin{proof}Let $\calM=\bigoplus_{i} \calM_i^{\oplus a_i}$ be an isotypic decomposition of $\calM$. Suppose that $\calM$ admits an $F^n$-structure $\Phi_\calM$ and write $V_\calM$ for the associated \DMs. Since $\Phi_\calM$ permutes the $\calM_i$'s, each of them admits an $F^{nm}$-structure for some $m>0$. Therefore, each $\calM_i$ admits a \DMs\ $V_{\calM_i}$. Let $(\calT_i,V_{\calT_i})$ be the maximal trivial subobject of $(\calM_i,V_{\calM_i})^\vee \otimes (\calM,V_\calM)$. By Lemma \ref{H0:l}, each $\calT_i$ is the maximal trivial subobject of $\calM_i^\vee \otimes \calM$. By construction, we have a tautological inclusion of $(\calM_i,V_{\calM_i})\otimes (\calT_i,V_{\calT_i})$ in $(\calM,V_\calM)$ for each $i$. This gives the desired decomposition. To prove the unicity it is enough to assume $\calM$ irreducible. In this case, if $V_{\calM}'$ is another \DMs, we have that $(\calM,V_{\calM})$ and $(\calM,V_\calM')$ are both irreducible and they admit a non-trivial morphism by Lemma \ref{H0:l}. \end{proof}
\begin{lemm}\label{soc-DM:l}
	If $(\calM,V_\calM)$ is an isocrystal with \DM\ $\Qpi$-structure, the socle $\calN\subseteq\calM$ admits a (unique) \DM\ \Qpis\ $V_\calN$ and $(\calN,V_\calN)$ is a subobject of $(\calM,V_\calM)$. Moreover, if $(\calM,V_\calM)$ is semi-simple, $\calM$ is semi-simple.
	
\end{lemm}

\begin{proof}Let $\Phi_\calM$ be an $F^n$-structure of $\calM$. Since $\Phi_\calM$ preserves $\calN$, it induces a \DMs\ $\calV_\calN$ of $\calN$ which makes $(\calN,V_\calN)$ a subobject of $(\calM,V_\calM)$. If $(\calM,\Phi_\calM)$ is irreducible, we deduce that $(\calN,V_\calN)=(\calM,V_\calM)$. Therefore, by Lemma \ref{ss-uniq-DM:l}, the isocrystal $\calN$ is irreducible. This concludes the proof.
\end{proof}

\begin{prop}\label{sq-DM:p}
If $(\calM,V_\calM)$ is an isocrystal with \DMs, each irreducible subquotient $\calM_i$ of $\calM$ admits a (unique) \DMs\ $V_{\calM_i}$ which makes $(\calM_i,V_{\calM_i})$ a subquotient of $(\calM,V_\calM)$.
\end{prop}

\begin{proof}
Let $\Phi_\calM$ be an $F^n$-structure of $\calM$ inducing $V_\calM$. By Lemma \ref{soc-DM:l}, after forgetting the \Qpis\ the socle filtration of $(\calM,V_\calM)$ is sent to the socle filtration of $\calM$ and the various steps of the socle filtration of $(\calM,V_\calM)$ have \DMs. Therefore, after taking the semi-simplification of $(\calM,V_\calM)$ with respect to this filtration, the result follows from Lemma \ref{ss-uniq-DM:l}.
\end{proof}

\begin{prop}\label{obse-ffur-forg:p}
The functor $\Lambda_\eta:\Fiisoc(X)\to \Isocl(X,\eta)$ is observable and, if $k=\ka$, it is fully faithful on unit-root objects.
\end{prop}
\begin{proof}
To prove that $\Lambda_\eta$ is observable it is enough to notice that by Proposition \ref{sq-DM:p} every rank $1$ subquotient of an isocrystal with \DMs\ is itself an isocrystal with \DMs. If $k=\ka$ unit-root $F^n$-isocrystals over $k$ are trivial by Dieudonné--Manin decomposition, which implies that $\Lambda_\eta$ is fully faithful thanks to Lemma \ref{H0:l}.
\end{proof}

\begin{defi}
We write $\Isoc_{\Qpu}(X,\eta)_F$ for the Tannakian subcategory of $\Isocl(X,\eta)$ spanned by the essential image of $\Lambda_\eta$.
\end{defi}


All we said works unchanged for overconvergent isocrystals and we have an observable functor $\Lambda_\eta:\Fioi(X)\to \oil(X,\eta)$ from overconvergent $F^\infty$-isocrystals to overconvergent isocrystals with $\Qpi$-structure (defined in the analogous way). We are now ready to present the fundamental result which compares the monodromy groups of $F^\infty$-isocrystals and isocrystals with \DMs.

\begin{defi}\label{Fi-mono:d} Let $(X,\eta)$ be a geometrically connected variety over $k$ endowed with a geometric point $\eta$. For an $F^n$-isocrystal $(\calM,\Phi_\calM)$ we write $G(\calM,\Phi_\calM^\infty,\eta)$ for the Tannaka group of $\langle \calM, \Phi_\calM^\infty\rangle$ with respect to $\omega_{\eta,\Qpi}$ and $G(\calM,V_\calM,\eta)$ for the Tannaka group of $\langle \calM, V_\calM\rangle$ with respect to $\omega_{\eta,\Qpi}$, where $V_\calM$ is the \DMs\ induced by $\Phi_\calM$. We also denote by $G(\calM,\Phi^{\infty}_\calM,\eta)^{\cst}$ the Tannaka group of \textit{constant} $F^\infty$-isocrystals in $\langle \calM, \Phi_\calM^\infty\rangle$, namely those $F^\infty$-isocrystals coming from $\Spec(k)$. We give analogous definitions in the overconvergent setting.
\end{defi}
\begin{prop}\label{fund-exac-seq:p} For a $\dag$-extendable $F^n$-isocrystal $(\calM,\Phi_\calM)$, we have the following commutative diagram of algebraic groups over $\Qpi$
	\begin{equation*}
		\begin{tikzcd}
			1\arrow{r} & G(\calM,V_\calM,\eta)\arrow{r}\arrow[hook,d] & G(\calM,\Phi^{\infty}_\calM,\eta)\arrow{r}\arrow[hook,d] & G(\calM,\Phi^{\infty}_\calM,\eta)^{\cst}\arrow{r}\arrow[d,two heads] &1\\
			1\arrow{r} & G(\calMd,V_{\calM},\eta)\arrow{r} & G(\calMd,\Phi^{\ag,\infty}_\calM,\eta)\arrow{r} & G(\calMd,\Phi^{\ag,\infty}_\calM,\eta)^{\cst}\arrow{r} & 1.
		\end{tikzcd}
	\end{equation*}
 The rows are exact, the first two vertical arrows are closed embeddings and the last arrow is a faithfully flat morphism.

\end{prop}

\begin{proof}
	The proof is similar to the proof of \cite[Cor. 5.12]{DE20}. To prove the exactness of the two sequences we use [\textit{ibid.}, Prop. A.13]. By Proposition \ref{obse-ffur-forg:p}, the functor $\langle \calM, \Phi_{\calM}^\infty \rangle\to \langle \calM,V_\calM\rangle$ is observable. Let us verify the second main property we need to apply the proposition. For $(\calN,\Phi_\calN^\infty)\in \langle \calM,\Phi_\calM^\infty\rangle$, we write $(\calN,V_\calN)$ for the associated isocrystal with $\Qpi$-structure. By Lemma \ref{H0:l}, if $\calT\subseteq \calN$ is the maximal trivial subobject, then $(\calT,V_\calN|_{\calT})$ is the maximal trivial subobject of $(\calN,V_\calN)$. By the maximality of $\calT$, we have that $(\calT,(\Phi_{\calN}|_\calT)^\infty)$ is a subobject of $(\calN,\Phi_\calN^\infty)$. This shows the exactness of the first row. One can argue similarly for the second row. The rest of the statement is a simple check. \end{proof}

\begin{rema}
Thanks to Theorem \ref{main:t} and Proposition \ref{rk1-unr:p} one can also show that the morphism $G(\calM,\Phi_\calM^\infty,\eta)^\cst\to G(\calMd,\Phi_\calM^{\ag,\infty},\eta)^\cst$ is an isomorphism.
\end{rema}

\subsection{Comparisons}

In this section we start by comparing the category $\Isocl(X,\eta)$ and the category $\Isoc(X)$ in order to relate the monodromy groups of isocrystals, as defined in \cite{Cre92a}, and the monodromy groups of isocrystals with \Qpis\ (Proposition \ref{comp:p}). In general, the scalar extension of the monodromy group of an isocrystal with \Qpis\ from $\Qpi$ to $K$ is bigger than the monodromy group of the associated isocrystal. For isocrystals with \DMs, we want to prove instead that we have the expected base change property. Subsequently, we want to compare isoclinic $F^\infty$-isocrystals (or rather direct sums of these) and lisse $\Qpu$-sheaves (Proposition \ref{etal-isoc:p}).

\begin{lemm}\label{ffss-comp:l}
The functor $\Isocl(X,\eta)_F\to \Isoc(X)$ is $K/\Qpi$-fully faithful (cf. \cite[Def. 1.1.2]{Sta08}) and sends semi-simple objects to semi-simple objects.
\end{lemm}

\begin{proof} We first show that the functor is $K/\Qpi$-fully faithful. This is equivalent to showing that for $(\calM,V_\calM)\in \Isocl(X,\eta)_F$, the maximal trivial subobject of $\calM$ can be upgraded to a trivial subobject of $(\calM,V_\calM)$. By Proposition \ref{obse-ffur-forg:p}, every object $(\calM,V_\calM)\in \Isocl(X,\eta)_F$ is a subobject of some $(\calN,V_\calN)\in \Isocl(X,\eta)_F$ with $V_\calN$ a \DMs. By Lemma \ref{H0:l}, if $(\calT,V_{\calT})$ is the maximal trivial subobject of $(\calN,V_\calN)$, then $\calT$ is the maximal trivial subobject of $\calN$. The intersection of $(\calM,V_\calM)$ and $(\calT,V_{\calT})$ in $(\calN,V_\calN)$ is then an isocrystal with $\Qpi$-structure with underlying isocrystal isomorphic to the maximal trivial subobject of $\calM$, as we wanted. The second part follows from Lemma \ref{ss-uniq-DM:l}.\end{proof}

\begin{prop}\label{comp:p}
	
	The functor $\Isocl(X,\eta)_{F}\otimes_{\Qpi} K\to \Isoc(X)_{F}$ is an equivalence of categories. In particular, for an $F^n$-isocrystal $(\calM,\Phi_\calM)$, if $V_\calM$ is the associated \DMs, we have $G(\calM,\eta)\simeq G(\calM,V_\calM,\eta)\otimes_{\Qpu}K$.
\end{prop}

\begin{proof}This follows from \cite[Thm. 2.4.1]{Sta08} thanks to Lemma \ref{ffss-comp:l}.\end{proof}

\begin{defi}
We write $\Fiisoce(X)$ for the Tannakian subcategory of $\Fiisoc(X)$ spanned by the isoclinic $F^\infty$-isocrystals and $\Isocle(X,\eta)$ for the Tannakian subcategory of \break $\Isocl(X,\eta)$ spanned by the image of the functor $\Fiisoce(X)\to \Isocl(X,\eta)$. We say that the category $\Isocle(X,\eta)$ is the category of \textit{étale isocrystals with $\Qpi$-structure}.
\end{defi}

\begin{prop}\label{etal-isoc:p}
If $k=\ka$, there is a natural equivalence $\LS(X,\Qpi)\iso \Isocle(X,\eta).$ 
\end{prop}

\begin{proof}
By \cite[Thm. 2.1]{Cre87}, for every $n$ there is a fully faithful $\Q_{p^n}$-linear $\otimes$-functor $\LS(X,\Q_{p^n})\hookrightarrow \Fnisoc(X)$ with essential image the category of unit-root $F^n$-isocrystals. This family of $\otimes$-functors induces a fully faithful $\Qpi$-linear $\otimes$-functor $\LS(X,\Qpi)\hookrightarrow \Fiisoc^\et(X)$. If we postcompose this functor with the functor $\Lambda_\eta:\Fiisoc^\et(X)\to \Isocle(X,\eta)$, we get by Proposition \ref{obse-ffur-forg:p} a fully faithful $\Qpi$-linear $\otimes$-functor $\LS(X,\Qpi)\hookrightarrow \Isocle(X,\eta)$. To prove that this functor is essentially surjective it is enough to notice that every isoclinic object in $\Fiisoc(X)$ is a tensor product of a unit-root $F^\infty$-isocrystal by a rank $1$ constant $F^\infty$-isocrystal.
\end{proof}
\begin{rema}
Similarly one can prove that $\Fiisoce(X)$ is equivalent to the category of $\Q$-graded lisse $\Qpi$-sheaves.
\end{rema}
\begin{coro}
Let $(\calM,\Phi_\calM)$ be an $F^n$-isocrystal over $X$ which admits the slope filtration. If $k=\ka$, an \textit{étale path}\footnote{We recall that an \textit{étale path} is an isomorphism between two fibre functors of the Galois category of finite étale covers of $X$ induced by geometric points, \cite[Exp. V, §7]{SGA1}.} $\gamma$ joining two geometric points $\eta$ and $\eta'$ induces isomorphisms $G(\calM,\Phi_\calM^\infty,\eta)\iso G(\calM,\Phi_\calM^\infty,\eta')$ and $G(\calM,V_\calM,\eta)\iso G(\calM,V_\calM,\eta')$.
\end{coro}
\begin{proof}
Write $(\calN,\Phi_\calN)$ for $\Gr_{S_\bullet}(\calM,\Phi_\calM)$ and $V_\calM$ for the \DM\ \Qpis\ of $(\calM,\Phi_\calM)$ at $\eta$. The $\Qpi$-structure $V_\calM$ is naturally isomorphic to the \DM\ \Qpis\ of $(\calN,\Phi_\calN)$ at $\eta$. Thanks to Proposition \ref{etal-isoc:p}, the étale path $\gamma$ induces an isomorphism between the fibre functors $\omega_{\eta,\Qpi}:\langle \calN, V_\calN\rangle\to \VVec_{\Qpi}$ and $\omega_{\eta',\Qpi}:\langle \calN,V_\calN\rangle\to \VVec_{\Qpi}$. This yields the desired result.
\end{proof}

 \section{Main theorems}\label{main-theo:s}
\subsection{\texorpdfstring{$\dag$}{}-hull of \texorpdfstring{$F$}{}-isocrystals}\label{d-hull:ss}
In order to prove the parabolicity conjecture we chiefly study the \textit{$\dagger$-hulls} of sub-$F^n$-isocrystals.
 \begin{defi}
Let $(\calN,\Phi_\calN)\subseteq (\calM,\Phi_\calM)$ be an inclusion of $F^n$-isocrystals. The \textit{\dhl\ }of $(\calN,\Phi_\calN)$ in $(\calM,\Phi_\calM)$ is the smallest\footnote{Note that the category of isocrystals is artinian because if $\calM'\subseteq \calM$ have the same ranks, then $\calM'=\calM$.} \dex\ subobject of $(\calM,\Phi_\calM)$ containing $(\calN,\Phi_\calN)$. We denote it by $(\caloN,\Phi_\caloN)$. 
 \end{defi}

 \begin{nota}
 	Let $(\calM,\Phi_\calM)$ be a \dex\ $F^n$-isocrystal over $X$ with constant slopes. We say that $(\calM,\Phi_\calM)$ satisfies $\MS(\calM,\Phi_\calM)$ (where $\MS$ stands for “minimal slope”) if for every sub-$F^n$-isocrystal  $(\calN,\Phi_\calN)\subseteq (\calM,\Phi_\calM)$, the isocrystals $S_1(\calN)$ and $S_1(\caloN)$ are the same. We also say that $X$ satisfies $\MS(X)$ if for every $n>0$ and every \dex\ $F^n$-isocrystal $(\calM,\Phi_\calM)$ over $X$, we have that $\MS(\calM,\Phi_\calM)$ is true.
 \end{nota}

Throughout §\ref{main-theo:s} we want to prove the following theorem.

 \begin{theo}\label{main:t} A smooth variety $X$ over a perfect field $k$ satisfies $\MS(X)$.
 \end{theo}
 
 
 We start with some reductions.
 \begin{lemm}\label{red-F:l}
	To prove $\MS(X)$ we may assume $n=1$.
\end{lemm}
\begin{proof}
	Let $(\calN,\Phi_\calN)\subseteq (\calM,\Phi_\calM)$ be an inclusion of $F^n$-isocrystals. This induces an inclusion $(\calN',\Phi_\calN')\subseteq (\calM',\Phi_\calM')$ of $F$-isocrystals, where $\calM':=\bigoplus_{i=0}^{n-1} (F^i)^*\calM$ and $\calN':=\bigoplus_{i=0}^{n-1} (F^i)^*\calN$. Since $F^*$ is an autoequivalence of the category of overconvergent isocrystals, we deduce that $\overline{\calN'}=\bigoplus_{i=0}^{n-1} (F^i)^*\caloN$. This shows that if $\MS(\calM',\Phi_\calM')$ is true, then $\bigoplus_{i=0}^{n-1} (F^i)^*\calN=S_1({\calN'})=S_1(\overline{\calN'})=\bigoplus_{i=0}^{n-1} (F^i)^*S_1(\caloN)$ (note that $S_1(-)$ commutes with $F^*$). In turn, this implies that $S_1(\calN)=S_1(\caloN)$, which yields the desired result.
\end{proof}

 \begin{lemm}\label{red-isoc:l}
 To prove $\MS(X)$ it is enough to check isoclinic sub-$F$-isocrystals.
 \end{lemm}
 
 \begin{proof}Suppose that $\MS(X)$ is true for isoclinic sub-$F$-isocrystals. Let $(\calM,\Phi_\calM)$ be a \dex\ $F$-isocrystal over $X$. We want to prove $\MS(\calM,\Phi_{\calM})$ by induction on the length of the slope filtration of the sub-$F$-isocrystals of $(\calM,\Phi_\calM)$. Let $(\calN,\Phi_\calN)\subseteq (\calM,\Phi_\calM)$ be a sub-$F$-isocrystal of length $m\geq 2$ and assume that we already know the statement for any subobject with slope filtration of length at most $m-1$. It is then enough to show that $S_1(\overline{S_{m-1}(\calN)})=S_1(\caloN)$, since this would imply that $S_1(\calN)=S_1(\caloN)$ because $S_1(\calN)=S_1(\overline{S_{m-1}(\calN)})$. If $\overline{S_{m-1}(\calN)}=\caloN$ there is nothing to prove. If $\overline{S_{m-1}(\calN)}\subsetneq\caloN$ we have to show that the minimal slope of $\caloN/\overline{S_{m-1}(\calN)}\neq 0$ is greater than the minimal slope of $\overline{S_{m-1}(\calN)}$, which by the inductive hypothesis is $s_1$. Write $\pi:\caloN\to \caloN/\overline{S_{m-1}(\calN)}$ for the natural projection. By definition, we have that the sub-$F$-isocrystal $\pi(\calN)\subseteq \caloN/\overline{S_{m-1}(\calN)}$ is isoclinic of slope $s_m$ and its $\dag$-hull is $\caloN/\overline{S_{m-1}(\calN)}$ itself. As we have assumed that $\MS(X)$ is true for isoclinic sub-$F$-isocrystals, we deduce that the minimal slope of $ \caloN/\overline{S_{m-1}(\calN)}$ is then equal to $s_m$, the slope of $\pi(\calN)$, which is strictly greater than $s_1$. This concludes the proof.
 \end{proof}
For next lemma we need the following theorem proven by Kedlaya.
 	\begin{theo}[Kedlaya]\label{shrink-zar:t}
 	Let $U\subseteq X$ be a dense open. The following statements are true.
 	\begin{itemize}
\item[{\normalfont(i)}]The restriction functor $\Fisoc(X)\to \Fisoc(U)$ is fully faithful.
\item[{\normalfont(ii)}]The restriction functor $\Foi(X)\to \Foi(U)$ is fully faithful and closed under the operation of taking subquotients.
 	\end{itemize}
	
\end{theo}
\begin{proof}
Point (i) follows from \cite[Thm. 5.2.1]{Ked07} and \cite[Thm. 4.2.1]{Ked08} as explained in \cite[Thm. 2.2.3]{DK16}, while point (ii) follows from \cite[Thm. 5.2.1 and Prop. 5.3.1]{Ked07}.
\end{proof}
 
 \begin{lemm}\label{shrink-zar:l}
 	If $U$ is a dense open of $X$, then $\MS(\calM|_{U},\Phi_{\calM|_U})$ implies $\MS(\calM,\Phi_\calM)$. 
 \end{lemm}
 \begin{proof} Let $(\calN,\Phi_\calN)$ be an isoclinic subobject of $(\calM,\Phi_\calM)$. We want to prove that the operation of taking $\dag$-hulls commutes with the restriction functor to $U$. We have by definition $\overline{(\calN|_U)}\subseteq \caloN|_U$. On the other hand, by Theorem \ref{shrink-zar:t}, the \dex\ isocrystal $\overline{(\calN|_U)}$ extends to some \dex\ isocrystal $\calM'$ over $X$ such that $\calN\subseteq\calM'\subseteq \calM$. Since $\caloN\subseteq \calM'$ by definition, we show that $\overline{(\calN|_U)}= \caloN|_U$. Thanks to this, if $S_1(\overline{\calN|_U})=\calN|_U$ then $S_1(\caloN)|_U=S_1(\overline{\calN|_U})=\calN|_U$. This implies that $S_1(\caloN)=\calN$.
 \end{proof}
 
 \begin{lemm}\label{shri-étal:l}
 	If $f:Y\to X$ is a finite étale Galois cover the following are true.
 	
 	\begin{itemize}

\item[{\normalfont (i)}] For every \dex\ $(\calM,\Phi_\calM)$ over $X$ with constant slopes $\MS(f^*\calM,f^*\Phi_{\calM})$ implies $\MS(\calM,\Phi_\calM)$.
\item[{\normalfont (ii)}] For every \dex\ $(\calM,\Phi_\calM)$ over $Y$ with constant slopes $\MS(f_*\calM,f_*\Phi_{\calM})$ implies $\MS(\calM,\Phi_\calM)$.

 	\end{itemize}
 	 
 \end{lemm}
 \begin{proof}
  As in Lemma \ref{shrink-zar:l}, for (i) it is enough to show that for an isoclinic subobject $(\calN,\Phi_\calN)\subseteq(\calM,\Phi_\calM)$ we have that $f^*\caloN=\overline{f^*\calN}$. Write $G$ for the Galois group of the cover. The inclusion $f^*\caloN\supseteq\overline{f^*\calN}$ follows from the definition of $\dag$-hull. On the other hand, if $\calM':=\overline{f^*\calN}$, then the intersection $\bigcap_{g\in G}g^*\calM'$ contains $f^*\calN$ and it is \dex, thus it is equal to $\calM'$. This implies that $(\calM',\Phi_{\calM'})$ descends to some \dex\ $F$-isocrystal over $X$ which contains $(\calN,\Phi_\calN)$ and is contained in $(\calM,\Phi_\calM)$. Therefore $f^*\caloN=\overline{f^*\calN}$, as we wanted.
 	
 	\spa
 	
 	We prove now (ii). Let $(\calN,\Phi_\calN)$ be an isoclinic subobject of $(\calM,\Phi_\calM)$. We have that $f^*f_*\calN=\bigoplus_{g\in G}g^*\calN\subseteq \bigoplus_{g\in G}g^*\calM=f^*f_*\calM$ so that $\overline{f^*f_*\calN}=\bigoplus_{g\in G}g^*\caloN$. Because of this, to prove that $S_1(\caloN)=\calN$ it is enough to prove that $S_1(\overline{f^*f_*\calN})=f^*f_*\calN.$ By the previous argument we have that $\overline{f^*f_*\calN}=f^*\overline{f_*\calN}$, so that $S_1(\overline{f^*f_*\calN})=f^*(S_1(\overline{f_*\calN}))$ (note that $S_1(-)$ commutes with $f^*$). On the other hand, by the assumption that $S_1(\overline{f_*\calN})=f_*\calN$ we have that $f^*(S_1(\overline{f_*\calN}))=f^*f_*\calN$ and this yields the desired result.
 \end{proof}

 \subsection{The case of curves}\label{curve-case:ss}
In this section we study $\MS(X)$ when $X$ is a curve. In this case, $\MS(X)$ has also been proven by Tsuzuki in \cite[Prop. 6.1]{Tsu19}. The proof we present here is a shorter variant where we avoid [\textit{ibid.}, Thm. 3.27]. The main ingredient in our case is de Jong's theorem on the existence of the \textit{reverse slope filtration} (also used by Tsuzuki) at the generic point of $X$. On the other hand, in Corollary \ref{PBS:c} we extend [\textit{ibid.}, Thm. 3.27] to arbitrary smooth varieties. The strategy is to first prove a version of $\MS(X)$ for the generic point and then pass to the global setting. Thanks to the reductions of §\ref{d-hull:ss}, it is enough to treat the case of the affine line.

\spa

 Consider the ring $\calO_{\calE}:=(W[t]_{(p)})^\wedge$, where $(-)^\wedge$ denotes the $p$-adic completion. This is a complete discrete valuation ring unramified over $W$ with residue field $k(t)$. Let $\calO_{\calEd}\subseteq \calO_{\calE}$ be the subring of functions which converge in some annulus $*\leq |t| < 1$. These two rings are both endowed with a Frobenius lift $\varphi(t)=t^p$ and a derivation $\partial_t$. Write $\calE$ and $\calEd$ for the respective fields of fractions. 
 
 \spa
 
 
 
 \begin{defi}[$(\varphi,\nabla)$-modules]\label{phi-nabla:d}
 	If $E$ is either $\calE$ or $\calEd$, we say that a finite dimensional vector space $M$ over $E$ is a \textit{$(\varphi,\nabla)$-module} if it is endowed with a $\varphi$-linear isomorphism $\varphi_M:M\iso M $ and an additive morphism $\nabla_{\partial_ t}:M\to M$ which satisfies the Leibniz rule and such that $\nabla_{\partial_t} \circ \varphi_M=pt^{p-1}\varphi_M\circ \nabla_{\partial_t}$.
 \end{defi}
 
 The category $\Fisoc(k(t))$ is the category of $(\varphi,\nabla)$-modules over $\calE$ and we denote by $\Foi(k(t))$ the category of $(\varphi,\nabla)$-modules over $\calEd$. Thanks to Theorem 5.1 \cite{KedFull}, the natural functor $$\Foi(k(t))\to \Fisoc(k(t))$$ is fully faithful.
 
 \begin{prop}[Kedlaya, Tsuzuki\footnote{ 
 		We first learned about a proof of Proposition \ref{local:p} from Kedlaya via a private communication. The proposition also corresponds essentially to \cite[Thm. 2.14]{Tsu19}.}]\label{local:p}
 	If $N\subseteq M$ is an inclusion of $(\varphi,\nabla)$-modules over $\calE$ and $M$ is \dex, then $S_1(N)=S_1(\overline{N})$.
 \end{prop}
 
 \begin{cons}\label{var-dhl:c}
	Let $Q^\ag$ be the image of the composition of the natural morphisms $$(M^\ag)^\vee:=\Hom_\calEd(M^{\ag},\calEd)\to\Hom_\calE(M,\calE)\to \Hom_\calE(N,\calE)=:N^\vee.$$ We have natural maps $$M^\vee=(M^{\ag})^\vee\otimes_{\calEd}\calE\twoheadrightarrow Q^\ag \otimes_{\calEd} \calE\twoheadrightarrow N^\vee.$$
	The first arrow is surjective by construction, the second one is surjective because the morphism $M^\vee\to N^\vee$ is surjective. Note that even though $Q^\ag\subseteq N^\vee$, the second map needs not to be injective. Dualising with respect to $\calE$ we get inclusions $N\subseteq Q^\vee \subseteq M$. 
\end{cons}

To prove Proposition \ref{local:p} we first need the following construction.

 \begin{lemm}\label{dhl:l}
 	The $(\varphi,\nabla)$-module $Q^\vee$ is the $\dag$-hull of $N$ in $M$. In other words, $\overline{N}$ is the unique submodule of $M$ which contains $N$ and comes from some $\overline{N}^\ag\subseteq M^\ag$ such that $(\overline{N}^\ag)^\vee\to N^\vee$ is injective.
 \end{lemm}
 \begin{proof}
 	By construction, $Q^\vee$ is \dex\ and it contains $N$, so that $\overline{N}\subseteq Q^\vee$. On the other hand, we have morphisms $(M^\ag)^\vee\twoheadrightarrow (\overline{N}^\ag)^\vee\to N^\vee$, where the first one is surjective. By definition, the morphism $(\overline{N}^\ag)^\vee\to N^\vee$ factors through $Q^\ag$, which implies that $Q^\vee\subseteq \overline{N}$.
 \end{proof}
 
 We recall now the reverse filtration introduced by de Jong in \cite[Prop. 5.5]{deJ98}. For this, we need to introduce two other discrete valuation fields lifting $k(t)^\alg$.
 \begin{defi}
 	Let $\calO_{\calE_{\alg}}$ be the ring of Witt vectors of $k(t)^\alg$ (which is contained in the ring $\Gamma_2=\Gamma_{2,1}$ in de Jong's notation). Every element of $\calO_{\tcalE}$ can be written uniquely as $\sum_{i=0}^\infty[f_i]p^i$ where $[f_i]$ is the Teichm\"uller lift of some $f_i\in k(t)^\alg$. Consider the subring $\calO_{\tcalE}^\ag\subseteq \calO_{\tcalE}$ of those series such that the $t$-adic valuations of $f_i$ are bounded below by some linear function in $i$. This subring is preserved by the Frobenius of $\calO_{\tcalE}$ (and it is contained in $\Gamma_{2,c}=\Gamma_{2,1,c}$ in de Jong's notation). We write $\tcalEd$ and $\tcalE$ for the fraction fields.
 \end{defi}

 \begin{theo}[de Jong]\label{op-slop:t}
 	For a $\varphi$-module ${M}^{\ag}_\alg$ over $\tcalEd$ the following statements are true. 
 	\begin{itemize}

 		\item[(i)]$M_{\alg}^{\ag}$ admits a reverse slope filtration, i.e. there exists a filtration $$0=S_{0}^{\rev}(M_{\alg}^{\ag})\subsetneq S^{\rev}_1(M_{\alg}^{\ag})\subsetneq \dots \subsetneq S^{\rev}_m(M_{\alg}^{\ag})=M_{\alg}^{\ag}$$ of $\varphi$-modules over $\tcalEd$ such that $(S^{\rev}_i(\tMd)/S^{\rev}_{i-1}(\tMd))\otimes_{\tcalEd}\tcalE$ is isomorphic to\break $S_{m-i}(M_{\alg})/S_{m-i-1}(M_{\alg})$.
 		\item[(ii)]If $M^\ag$ is isoclinic of slope $s/r$, after possibly multiplying $s$ and $r$ by some positive integer, the $\varphi$-module $M_{\alg}^{\ag}[p^{1/r}]$ admits a basis of vectors $\{v_1,\dots,v_d\}$ such that $\varphi(v_i)=p^{s/r}v_i$.
 	\end{itemize}
 	
 \end{theo}
\begin{proof}
The result is proved in {\cite[Prop. 5.5]{deJ98}} for $\varphi$-modules over the bigger field $\Gamma_{2,c}[\tfrac 1p]$. To extend it in our setting, it is enough to notice that in [\textit{ibid.}, Prop. 5.1] one can replace $\Gamma_{2,b,c}$ by $\calO_{\calE_{\alg}}^\dag[p^{1/b}]$ since the residue field of $\calO_{\calE_{\alg}}^\dag[p^{1/b}]$ is algebraically closed and $\calO_{\calE_{\alg}}[p^{1/b}]\cap \Gamma_{2,b,c}=\calO_{\calE_{\alg}}^\dag[p^{1/b}]$. The rest of the proof works unchanged.
\end{proof}

 \begin{lemm}\label{local-K:l}
 	Let $M^\ag$ be a $\varphi$-module over $\calEd$ and let $N$ be an isoclinic $\varphi$-module over $\calE$ of slope $s/r$. For every morphism $\psi:M\to N$ of $\varphi$-modules, if the restriction of $\psi$ to $M^\ag$ is injective, then the maximal slope of $M$ is $s/r$ and the rank of $S_m(M)/S_{m-1}(M)$ is smaller or equal than the rank of $N$.
 \end{lemm}
 \begin{proof}This is a variant of \cite[Lem. 4.2]{KedFull}. Since $\calE_{\alg}^\ag$ is flat over $\calEd$ and $\calE\otimes_{\calEd}\tcalEd\to\tcalE$ is injective by \cite[Prop. 4.1]{KedFull}, then $\psi|_{M^\ag}$ induces an injective morphism
 	
 	$$\psi':\tMd:=M^\ag\otimes_{\calEd}\tcalEd\to N\otimes_{\calEd}\tcalEd\to  N\otimes_{\calE}\tcalE.$$	
 	The restriction of $\psi'$ to $S^{\rev}_1(M_{\alg}^{\ag})$ induces a non-trivial morphism $$S^{\rev}_1(M_{\alg}^{\ag})\otimes_{\tcalEd}\tcalE\to N\otimes_{\calE}\tcalE.$$ This implies that the slope of $S^{\rev}_1(M_{\alg}^{\ag})$, which is the maximal slope of $M$, is $s/r$. Moreover, by Theorem \ref{op-slop:t}.(ii), after possibly enlarging $r$, the dimension of the $\Qp(p^{1/r})$-vector space $(S^{\rev}_1(M_{\alg}^{\ag})[p^{1/r}])^{\varphi=p^{s/r}}$ is equal to the rank of $S_m(M)/S_{m-1}(M)$. Similarly, by the Dieudonné--Manin decomposition, $(N\otimes_{\calE}\tcalE[p^{1/r}])^{\varphi=p^{s/r}}$ is a $\Qp(p^{1/r})$-vector space of dimension equal to the rank of $N$. We then obtain the inequality of ranks thanks to the injectivity of $\psi'$.
 \end{proof}
\subsubsection{Proof of Proposition \ref{local:p}.}
 	By Lemma \ref{red-isoc:l}, it is enough to prove the result when $N$ is isoclinic of slope $s/r$ and $\overline{N}=M$. In that case, we have to check that $\overline{N}$ has minimal slope $s/r$ and that the inclusion $N\subseteq S_1(\overline{N})$ is an equality. By Lemma \ref{dhl:l}, the morphism ${\overline{N}}^\vee\to N^\vee$ satisfies the assumptions of Lemma \ref{local-K:l}, thus $\overline{N}$ has minimal slope $s/r$. Moreover, since $$\rk(N)=\rk(N^\vee)\geq\rk( S_m(\overline{N}^\vee)/S_{m-1}(\overline{N}^\vee))=\rk(S_1(\overline{N})),$$ we get $N=S_1(\overline{N})$. 	
 	\qed

 \subsubsection{}
Next step is to pass from the result on the generic point to the global situation. Thanks to the reductions in §\ref{d-hull:ss}, it is enough to work with $F$-isocrystals over $\mathbb{A}^1_k$. Let $K\langle u \rangle$ and $K\langle u \rangle^\ag$ be the rings of convergent and overconvergent series over $K$. The category $\Fisoc(\mathbb{A}^1_k)$ is the category of $(\varphi,\nabla)$-modules over $K\langle u \rangle$ (defined as in §\ref{phi-nabla:d}) while $\Foi(\mathbb{A}^1_k)$ is the category of $(\varphi,\nabla)$-modules over $K\langle u \rangle^\ag$. We consider the ring homomorphisms $K\langle u \rangle\to \calE$ and $K\langle u \rangle^\ag\to \calE^\ag$ sending $u\mapsto \tfrac 1t$. 

\spa

Let $A_n$ be the image of the morphism $W\langle u_1,\dots, u_n\rangle \to W\langle u \rangle $ which sends $u_i\mapsto \bar{u}_i:=pu^i$ and let $\mathfrak{m}_n\subseteq A_n[t]$ be the maximal ideal $(p,t,\bar{u}_1,\dots,\bar{u}_n)$. We denote by $B_n\subseteq \calO_{\calE}$ the image of $(A_n[t])_{\mathfrak{m}_n}\to \calO_{\calE}$. The rings $A_n$ and $B_n$ provide integral models of the rings of overconvergent series. More precisely, we have that $\varinjlim_n {A}_n[\tfrac 1p]=K\langle u \rangle^\ag$ and $\varinjlim_n ({B}_n)^\wedge[\tfrac 1p]=\calEd$. Note that the $A_n$-algebra $B_n$ can be also written as $$\left( \tfrac{A_n[t]}{(\bar{u}_1t-p,\dots,\bar{u}_nt^n-p)}\right)_{\mathfrak{m}_n}.$$ It will be convenient for us to consider also the $A_n$-algebras $B_n':=\tfrac{A_n[t]}{(\bar{u}_1t-p)}$ and $\tilde{B}_n:=\left(B_n'\right)_{\mathfrak{m}_n}$. Note that $\tilde{B}_n$ admits a surjection $\tilde{B}_n\twoheadrightarrow B_n$ whose kernel is killed by $p^{n-1}$ since $$(\bar{u}_1t-p)|(\bar{u}_1^it^i-p^i)=p^{i-1}(\bar{u}_it^i-p).$$ This implies that $\varinjlim_n (\tilde{B}_n)^\wedge[\tfrac 1p]=\calEd$ as well.

 \begin{lemm}\label{bas:l}For every $n>0$, the natural morphism $(W\langle u \rangle {{\widehat{\otimes}}}_{{A}_n}{\tilde{B}_{n}})^\wedge [\tfrac 1p]\to \calE$ is injective\footnote{One can actually prove that it is surjective as well, but we do not need this here.}. 
 \end{lemm}
 
 \begin{proof}
 	
 	To prove the result it is enough to show that for every $e>0$, we have that $W_e[u] {{{\otimes}}}_{{A}_n}{\tilde{B}_{n}}\to \calO_{\calE}/p^{e}$ is killed by $p$. In addition, since localising at $\mathfrak{m}_n$ is exact and $(\calO_{\calE}/p^{e})_{\mathfrak{m}_n}=\calO_{\calE}/p^{e}$, we may replace $\tilde{B}_{n}$ with $B_n'$. After these reductions, it remains to prove that the kernel of the morphism 
 	$$W_e[u] {{{\otimes}}}_{{A}_n}{{B}'_{n}}=\tfrac{W_e[u,t]}{p(ut-1)}\to W_e[t]_{(p)}= \calO_{\calE}/p^{e}$$
 	is killed by $p$, which follows from the fact that the kernel is generated by $(ut-1)$.
 \end{proof}
\begin{lemm}\label{bounded-p-torsion:l}
If $M$ is an $A_n$-module such that $M[p^\infty]=M[p^s]$ for some $s\geq 0$, then $\mathrm{Tor}_1^{A_n}({B}'_n,M)$ is killed by $p^s$.
\end{lemm}
\begin{proof}
	The $A_n$-algebra ${B}_n'$ admits the free resolution $A_n[t]\xrightarrow{\cdot (\bar{u}_1t-p)}A_n[t]$, so that $\mathrm{Tor}_1^{A_n}({B}'_n,M)=\ker(M[t]\xrightarrow{\cdot (\bar{u}_1t-p)}M[t])$. Since ${B}_n'[\tfrac 1p]$ is a localisation of $A_n[\tfrac 1p]$, we have that $$\mathrm{Tor}_1^{A_n[\tfrac 1p]}({B}_n'[\tfrac 1p],M[\tfrac 1p])=0,$$ which implies that $\mathrm{Tor}_1^{A_n}({B}_n',M)\subseteq (M[t])[p^\infty]$. We end the proof by noticing that since $M[p^\infty]$ is killed by $p^s$, the same is true for $(M[t])[p^\infty]$.
\end{proof}

 \begin{prop}\label{l-to-g:p}
 	
 	Let $N\subseteq M$ be an inclusion of $(\varphi,\nabla)$-module over $K\langle u \rangle$, where $M$ is \dex\ and has constant slopes. The $(\varphi,\nabla)$-module $\overline{N}\otimes\calE$ is the \dhl\ of $N\otimes \calE$ in $M\otimes \calE$.
 \end{prop}
 \begin{proof}
 	The $K\langle u \rangle$-module $\overline{N}$ has a similar description as in §\ref{var-dhl:c}, where $\calE$ and $\calEd$ are replaced by the rings $K\langle u \rangle$ and $K\langle u \rangle^\ag$. By the analogue of Lemma \ref{dhl:l}, the $(\varphi,\nabla)$-module $\overline{N}$ comes from a $(\varphi,\nabla)$-module $\overline{N}^{\ag}$ over $K\langle u \rangle^\ag$ such that $(\overline{N}^{\ag})^\vee\to N^\vee$ is injective. Write $Q^\ag$ for $(\overline{N}^{\ag})^\vee$ and $P$ for $N^\vee$. Thanks to Lemma \ref{dhl:l}, it is enough to prove that $$Q^\ag\otimes_{K\langle u \rangle^\ag}\calE^\ag\to P\otimes_{K\langle u \rangle}\calE$$
 	is injective.
 	
 	\spa

 	By Prop 6.6 and Thm. 6.7 of \cite{KedFull}, both $P$ and $Q^\ag$ are free. Therefore, there exists a finite free $W\langle u \rangle$-module $P_W$ and a finite free $A_1$-module $Q_{W,1}$ such that $P=P_W[\tfrac{1}{p}]$ and $Q^{\ag}=\varinjlim_nQ_{W,n}[\tfrac{1}{p}]$ where $Q_{W,n}:=Q_{W,1}\otimes_{A_1}A_n$. Even the morphism $Q^\ag\to P$, after clearing the denominators, is induced by a compatible family of injective morphisms $f_n:Q_{W,n}\to P_W$. We note that on the one hand $$Q^\ag\otimes_{K\langle u \rangle^\ag}\calE^\ag=\varinjlim_n(Q_{W,n}\otimes_{A_n} \tilde{B}_n)^\wedge[\tfrac 1p],$$
 	and on the other hand, by Lemma \ref{bas:l}, $$(P_W\otimes_{A_n} \tilde{B}_n)^\wedge[\tfrac 1p]\subseteq P\otimes_{K\langle u \rangle}\calE$$ for every $n>0$.

 	\spa
 	We fix now $n>0$ and we endow the modules $Q_{W,n}$ and $P_W$ with the sup norms associated to some choices of basis. Since $Q_{W,n}$ is compact for the induced metric and $f_n$ is injective, we have that $$\inf\left\{ \|f_n(m)\|\  \big| \ {m\in Q_{W,n}\setminus p  Q_{W,n}}\right\}=p^{-s_n}$$ for some $s_n\geq 0$. This means that $p^{s_n+i}P_W\cap f(Q_{W,n})\subseteq p^if(Q_{W,n})$ for every $i\geq 0$ and this implies that the $p^\infty$-torsion of the quotient $R_{W,n}:=P_W/f(Q_{W,n})$ is killed by $p^{s_n}$. Therefore, by Lemma \ref{bounded-p-torsion:l}, the group $\mathrm{Tor}^{A_n}_1({B}_n',R_{W,n})$ is killed by $p^{s_n}$, which in turns implies that the kernel of $$Q_{W,n}\otimes_{A_n} {B}'_n\to P_W\otimes_{A_n} {B}'_n$$ is killed by $p^{s_n}$ as well. After localising at $\mathfrak{m}_n$, taking the $p$-adic completion, and inverting $p$, we show that 
 	$$(Q_{W,n}\otimes_{A_n} \tilde{B}_n)^\wedge[\tfrac{1}{p}]\to (P_W\otimes_{A_n} \tilde{B}_n)^\wedge[\tfrac{1}{p}]$$ is injective.
 	
 	\spa
 	
 	Knowing the injectivity for every $n>0$, we deduce that $$Q^\ag\otimes_{K\langle u \rangle^\ag}\calE^\ag=\varinjlim_n(Q_{W,n}\otimes_{A_n} \tilde{B}_n)^\wedge[\tfrac 1p]\to (P_W\otimes_{A_n} \tilde{B}_n)^\wedge[\tfrac 1p]\subseteq P\otimes_{K\langle u \rangle}\calE$$ is injective. This yields the desired result.\end{proof}
 
 \begin{theo}[See also {\cite[Prop. 6.1]{Tsu19}}]\label{main-curv:t}
 	A smooth curve $X$ over a perfect field $k$ satisfies $\MS(X)$.
 \end{theo}
\begin{proof}By \cite[Thm. 1]{Ked05}, there exists a dense open of $X$ which admits a finite étale cover to $\mathbb{A}^1_k$. Thanks to Lemma \ref{shrink-zar:l} and Lemma \ref{shri-étal:l}, this implies that we may assume $X=\mathbb{A}^1_k$. Let $N\subseteq M$ be an inclusion of $(\varphi,\nabla)$-module over $K\langle u \rangle$, where $M$ is \dex\ and has constant slopes. By Proposition \ref{l-to-g:p}, we have that $\overline{N}\otimes\calE$ is the \dhl\ of $N\otimes \calE$ in $M\otimes \calE$. Therefore, by Proposition \ref{local:p}, $$S_1(N)\otimes \calE=S_1(N\otimes \calE)=S_1(\overline{N}\otimes\calE)=S_1(\overline{N})\otimes{\calE}.$$
	This implies that $S_1(N)$ and $S_1(\overline{N})$ have the same slope and the same rank, so that $S_1(N)=S_1(\overline{N})$. This concludes the proof.
\end{proof}

 \subsection{Chevalley theorem and filtrations}\label{Chevalley:ss}
 
 In this section Hypothesis \ref{hypo} is in force. Let $\eta\in X(\Omega)$ be a perfect point of $X$ and let $(\calMd,\Phi_\calM^\ag)$ be an overconvergent $F^n$-isocrystal with constant slopes. Consider the Dieudonné--Manin fibre functor $\omega_{\eta,\Qpu}:\langle\calM,\Phi_\calM^\infty\rangle \to \VVec_{\Qpu}$ associated to $\eta$.
 Write $G$ for $G(\calMd,\Phi_\calM^{\ag,\infty},\eta)$ and $H$ for $G(\calM,\Phi_\calM^\infty,\eta)$. 
 \begin{defi}\label{coch:d}
For every $e$, let $\Gm^{1/e}$ be the torus with character group $\tfrac{1}{e}\Z$ and $\Gm^{1/\infty}:=\varprojlim_e \Gm^{1/e}$. If $r$ is the lcm of the denominators of the slopes of $\calM$, for every $(\calM',\Phi_{\calM'}^\infty)\in\langle\calM,\Phi_\calM^\infty\rangle$ we denote by $\tilde{S}_{s/r}(\omega_{\eta,\Qpu}(\calM'_\eta,\Phi_{\calM'}^\infty))\subseteq \omega_{\eta,\Qpu}(\calM'_\eta,\Phi_{\calM'}^\infty)$ the $\Qpu$-linear vector subspace of slope at most $s/r$. This defines an exact $\otimes$-filtration $\tilde{S}_{\bullet}$ of $\omega_{\eta,\Qpu}$ indexed by $\tfrac{1}{r}\Z$, that in turn defines a morphism $\lambda:\Gm^{1/\infty}\twoheadrightarrow \Gm^{1/r}\to G$ (cf. \cite[§2.1.1, page 213]{Saa72}). We say that $\lambda$ is the \textit{quasi-cocharacter} attached to the slope filtration of $\calM_\eta$. We denote by $P_G(\lambda)$ the subgroup of $G$ of those $\otimes$-automorphisms of $\omega_{\eta,\Qpu}$ preserving $\tilde{S}_{\bullet}$ (as in [\textit{ibid.}, §2.1.3, page 216]).

 \end{defi}

 \begin{prop}\label{para:p}If $\MS(X)$ is true then $H=P_G(\lambda)$. In particular, if $G$ is a reductive group then $H$ is a parabolic subgroup of $G$.
 \end{prop}
 \begin{proof}Since $H\subseteq P_G(\lambda)$, we have to prove that $P_G(\lambda)\subseteq H$. By Chevalley's theorem, there exists an overconvergent $F^\infty$-isocrystal $(\calNd,\Phi_{\calN}^{\ag,\infty})\in \langle \calMd, \Phi_\calM^{\ag,\infty} \rangle$ and a rank $1$ sub-$F^\infty$-isocrystal $(\calL,\Phi_\calL^\infty)\subseteq (\calN,\Phi_\calN^\infty)$, such that $H$ is the stabiliser of the line $$L:=\omega_{\eta,\Qpi}(\calL,\Phi_\calL)\subseteq \omega_{\eta,\Qpi}(\calN,\Phi_\calN):=V.$$ We have to prove that $P_G(\lambda)$ stabilises $L$. Let $(\caloL,\Phi_{\caloL})$ be the $\dag$-hull of $(\calL,\Phi_\calL)\subseteq (\calN,\Phi_\calN)$ and write $\overline{L}$ for $\omega_{\eta,\Qpi}(\caloL,\Phi_{\caloL})$. We denote by $s$ the slope of $(\calL,\Phi_\calL)$ and by $V^{\leq s}\subseteq V$ the subspace of slope smaller or equal than $s$. Since $\MS(X)$ is satisfied, we know that $\calL=S_1(\calL)=S_1(\caloL)$, which implies that $L=\overline{L}\cap V^{\leq s}.$ Since $\overline{\calL}\subseteq \calN$ admits by definition a $\dag$-extension, $P_G(\lambda)$ stabilises $\overline{L}$. On the other hand, $P_G(\lambda)$ stabilises $V^{\leq s}$ because $(\calNd,\Phi_{\calN}^{\ag,\infty})$ is an element in $\langle \calMd, \Phi_\calM^{\ag,\infty} \rangle$. This implies that $P_G(\lambda)$ stabilises $L$, thus $P_G(\lambda)\subseteq H$ as we wanted. If $G$ is reductive, $H=P_G(\lambda)$ is parabolic by \cite[Prop. 2.2.5, page 223]{Saa72}.
 \end{proof}

\subsubsection{}
We end this section by presenting some consequences of Proposition \ref{para:p}. Let $X$ be a smooth geometrically connected variety over $k$ such that every dense open $U\subseteq X$ satisfies $\MS(U)$, let $\eta\in X(\Omega)$ be a perfect point of $X$, and let $(\calM^\ag,\Phi_\calM^\ag)$ be an overconvergent $F^n$-isocrystal over $X$.
 \begin{prop}\label{pi0:p}
There is a canonical isomorphism $\pi_0(G(\calM,\Phi^\infty_\calM,\eta))=\pi_0(G(\calMd,\Phi^{\ag,\infty}_\calM,\eta))$.
 \end{prop}
 \begin{proof}
 		To prove the statement we may assume that $(\calMd,\Phi^{\ag,\infty}_\calM)$ is semi-simple by replacing it with the semi-simplification with respect to a Jordan--H\"older filtration. Next step is to show that the map $\pi_0(G(\calM,\Phi^\infty_\calM,\eta))\to\pi_0(G(\calMd,\Phi^{\ag,\infty}_\calM,\eta))$ is surjective. For this we do not need the previous proposition. Write $\langle\calMd,\Phi_\calM^{\ag,\infty}\rangle_{\mathrm{fin}}$ and $\langle\calM,\Phi_\calM^\infty\rangle_{\mathrm{fin}}$ for the Tannakian subcategories of objects with finite monodromy groups. The functor $\langle\calMd,\Phi_\calM^{\ag,\infty}\rangle\to \langle\calM,\Phi_\calM^\infty\rangle$ is fully faithful, thus the functor $\langle\calMd,\Phi_\calM^{\ag,\infty}\rangle_{\mathrm{fin}}\to \langle\calM,\Phi_\calM^\infty\rangle_{\mathrm{fin}}$ is fully faithful and observable, which implies that $\pi_0(G(\calM,\Phi^\infty_\calM,\eta))\to\pi_0(G(\calMd,\Phi^{\ag,\infty}_\calM,\eta))$ is surjective.
 		
 		\spa	
 		
 		Let $U$ be a dense open of $X$ where $(\calMd,\Phi_\calM^\ag)$ has constant slopes. Up to replacing $\eta$ we may assume that it lies in $U$. We have the following commutative diagram
 		\begin{equation*}
		\begin{tikzcd}
	 \pi_0(G(\calM|_U,\Phi^\infty_\calM|_U,\eta))\arrow[r,two heads]\arrow[d,two heads] & \pi_0(G(\calM,\Phi_\calM^\infty,\eta))\arrow[d,two heads]\\
 \pi_0(G(\calMd|_U,\Phi^{\ag,\infty}_\calM|_U,\eta))\arrow[r,"\sim"] & \pi_0(G(\calMd,\Phi^{\ag,\infty}_\calM,\eta)).
\end{tikzcd}
 		\end{equation*}	
 Thanks to Theorem \ref{shrink-zar:t}, the upper arrow is surjective while the lower one is an isomorphism. To prove the final statement it is enough to prove that the morphism $\pi_0(G(\calM|_U,\Phi^\infty_\calM|_U,\eta))\to\pi_0(G(\calMd|_U,\Phi^{\ag,\infty}_\calM|_U,\eta))$ is injective, or equivalently that $$H':=G(\calMd|_U,\Phi^{\ag,\infty}_\calM|_U,\eta)^\circ \cap G(\calM|_U,\Phi^\infty_\calM|_U,\eta)$$ is connected where $G(\calMd|_U,\Phi^{\ag,\infty}_\calM|_U,\eta)^\circ$ is the neutral component of $G(\calMd|_U,\Phi^{\ag,\infty}_\calM|_U,\eta)$. Indeed, this would imply that all the morphisms of the diagram are isomorphisms. To prove this we note that by Proposition \ref{para:p}, the group $H'$ is a parabolic subgroup of the connected reductive group $G(\calMd|_U,\Phi^{\ag,\infty}_\calM|_U,\eta)^\circ$, thus it is connected by \cite[Thm. 11.16]{Bor91}, as we wanted.
 \end{proof}
 
 \begin{prop}\label{rk1-unr:p}
 If $(\calL,\Phi_\calL^\infty)\in \langle \calM,\Phi_\calM^\infty\rangle$ is a $\dag$-extendable rank $1$ $F^\infty$-isocrystal, then $(\calLd,\Phi_\calL^{\ag,\infty})$ is in $\langle\calMd,\Phi_\calM^{\ag,\infty}\rangle$.
 \end{prop}
 \begin{proof}By Theorem \ref{shrink-zar:t} we may shrink $X$ and assume that $(\calM,\Phi_\calM)$ admits the slope filtration. Moreover, we may assume without loss of generality that $(\calMd,\Phi_\calM^\ag)$ is semi-simple. Let us write $G$ for $G(\calMd \oplus \calLd,\Phi_{\calM}^{\ag,\infty}\oplus\Phi_{\calL}^{\ag,\infty})$ and $G'$ for $G(\calMd,\Phi_\calM^{\ag,\infty})$. The inclusion $\langle\calMd,\Phi_{\calM}^{\ag,\infty} \rangle \subseteq \langle\calMd\oplus \calLd,\Phi_{\calM}^{\ag,\infty} \oplus \Phi_{\calL}^{\ag,\infty} \rangle$ induces a surjective morphism $f:G\twoheadrightarrow G'$. We want to prove that $N:=\ker(f)=1$. Let $P\subseteq G$ be the subgroup associated to $\langle\calM \oplus \calL,\Phi_\calM \oplus \Phi_\calL \rangle$. By Proposition \ref{para:p}, since we are assuming that $\MS(X)$ is true, the subgroup $P$ is a parabolic subgroup of $G$. Therefore, it contains a maximal torus $T$ of $G$. Since $\langle \calM,\Phi_{\calM}^\infty \rangle=\langle \calM \oplus \calL, \Phi_{\calM}^\infty\oplus\Phi_{\calL}^\infty\rangle$, the morphism $f$ is an isomorphism when restricted to $P$, so that $N\cap P=1$. Therefore, $N$ is a finite group, because it is a normal subgroup of the reductive group $G$ which intersects trivially the maximal torus $T$. The morphism $\pi_0(f)$ is an isomorphism by Proposition \ref{pi0:p}, because both $\pi_0(G)$ and $\pi_0(G')$ are equal to $\pi_0(P)$. This implies that $N$ is contained in $G^\circ$, the neutral component of $G$. Moreover, since $N$ is a finite normal subgroup, it is contained in $Z(G^\circ)\subseteq T$. This shows that $N=1$, as we wanted. \end{proof}
 
 \begin{coro}\label{inters:c}The group $G(\calM,\eta)$ is equal to the intersection $G(\calM^\ag,\eta)\cap G(\calM,\Phi^\infty_\calM,\eta)\otimes_{\Qpu} K$.
\end{coro}
\begin{proof} Since constant $F^\infty$-isocrystals are $\dag$-extendable, Proposition \ref{rk1-unr:p} implies that $$G(\calM,\Phi_\calM^\infty,\eta)^{\cst}= G(\calMd,\Phi_\calM^{\ag,\infty},\eta)^{\cst}.$$ We deduce that the left square of the diagram in Proposition \ref{fund-exac-seq:p} is cartesian. To get the final result we extend the scalars of the cartesian square from $\Qpu$ to $K$. Indeed, thanks to Proposition \ref{comp:p}, we know that $G(\calM,V_\calM,\eta)\otimes_{\Qpu} K=G(\calM,\eta)$ and $G(\calM^\ag,V_\calM,\eta)\otimes_{\Qpu} K=G(\calM^\ag,\eta)$. 
\end{proof}



 \subsection{Lefschetz theorem and proof of Theorem \ref{i-para-conj:t}}\label{Lefschet:ss}
 
 In this section we deal with the problem of reducing Theorem \ref{main:t} to the case of curves. For this aim, we prove a new Lefschetz theorem for overconvergent isocrystals (Theorem \ref{tame-lef:t}). Subsequently, we use Theorem \ref{main:t} to prove Theorem \ref{i-para-conj:t}.

 \begin{defi}
 	If $Y$ is smooth and proper variety, $D\subseteq Y$ is a simple normal crossing divisor and $X:=Y\setminus D$, we say that an overconvergent isocrystal over $X$ is \textit{docile} if it has unipotent monodromy along $D$ (cf. \cite[Def. 4.4.2]{Ked07}). We denote by $\Isoc^\ag(X)_\mathrm{doc}$ the category of docile overconvergent isocrystals over $X$.
 \end{defi}
 
 
 In \cite[Cor. 2.4]{AE19} the authors proved the following Lefschetz theorem for docile overconvergent isocrystals over perfect fields.
 
 \begin{theo}[Abe--Esnault]\label{AE-lef:t}
 	Let $Y\subseteq \PP^d_{k}$ be a smooth connected projective variety of dimension $\geq2$ and $D$ a simple normal crossing divisor. For every smooth curve $\overline{C}\subseteq Y$ which is a complete intersection of hypersurfaces intersecting transversally $D$, if we write $X$ for $Y\setminus D$ and $C$ for $\overline{C}\setminus D$, the functor $$\Isoc^\ag(X)_{\mathrm{doc}}\to \Isoc^\ag(C)_{\mathrm{doc}}$$ is fully faithful.
 \end{theo}
 We explain here how to refine their result and prove the following theorem. 
 
 \begin{theo}\label{tame-lef:t}
 	Let $Y\subseteq \PP^d_{\ka}$ be a smooth connected projective variety of dimension $\geq2$ and $D$ a simple normal crossing divisor. If $(\calMd, \Phi_\calM^\ag)$ is an overconvergent $F^n$-isocrystal over $X:=Y\setminus D$ docile with respect to $D$, then there exists a smooth connected curve $C\subseteq X$ intersecting the dense open of $X$ where the slopes of $(\calMd, \Phi_\calM^\ag)$ are constant and such that the restriction functor $\langle \calMd\rangle \to \langle \calMd|_{C} \rangle$
 	is an equivalence of categories.
 \end{theo}

\begin{rema}
In \cite[Thm. 3.10]{AE19} the authors improve Theorem \ref{AE-lef:t} to a stronger statement when $k$ is finite. Their proof does not extend to algebraically closed fields (not even $\ka=\F$) since they crucially use a finiteness result for rank 1 overconvergent $F$-isocrystals, which ultimately relies on class field theory.
\end{rema}

Before proving Theorem \ref{tame-lef:t} let us present some preliminary results that we will use during the proof. First, we need the following two lemmas on lisse sheaves.
 
 \begin{lemm}[Pink, Serre]\label{isoc-Lef:l}
 	
 	Let $X$ be a smooth connected variety over $k$ and $\calV$ a lisse $\Qpbar$-sheaf over $X$. There exists a connected finite Galois étale cover $\tilde{X}\to X$ (which depends on $\calV$) with the property that for every smooth connected curve $C\subseteq X$ such that $\tilde{X}\times_X C$ is connected, the restriction functor $$\langle\calV \rangle\to \langle \calV|_C \rangle$$
 	is an equivalence of categories. 
 	
 \end{lemm}
 \begin{proof}Let $\Pi$ be the image of the representation $\rho:\pi_1^\et(X,\eta)\to \GL(\calV_\eta)$ associated to $\calV$, where $\eta$ is a geometric point of $X$. By \cite[§10.6]{Ser89} or \cite[Key Lemma 8.18.3]{Kat90}, there exists a finite quotient $\pi:\Pi\twoheadrightarrow \bar{\Pi}$ with the property that for every profinite subgroup $H\subseteq \pi_1^\et(X,\eta)$, we have that $\rho(H)=\Pi$ if and only if $\pi(\rho(H))=\bar{\Pi}$. The Galois cover $\tilde{X}\to X$ induced by $\pi\circ \rho$ satisfies then the required property.
 \end{proof}

\begin{lemm}\label{ramif:l}Let $X$ be a smooth connected variety over $k$, $D\subseteq X$ an irreducible divisor and $\calV$ a lisse $\Qpbar$-sheaf over $X\setminus D$. There exists a dense smooth open $D'\subseteq D$ and a conic closed subscheme $Z\subseteq TX\times_X D'$ of codimension $1$ at every fibre which satisfies the following property.
	\spa
	
	\begin{quote}$\mathrm{R}_\calV(Z)$: Let $C\subseteq X$ be a smooth curve not contained in $D$ and intersecting $D'$ at some closed point $x$ such that $TC_x$ is not contained in $Z_x$. For every rank $1$ lisse sheaf $\calL\in \langle\calV\rangle$ ramified at $D$, $\calL|_C$ is ramified at $x$. 
	\end{quote}
	
\end{lemm}
\begin{proof}This result is a variant of \cite[Lem. 5.1]{Dri12}. Since the monodromy group of $\calV$ is of finite type, the group $\mathbb{X}$ of rank $1$ objects in $\langle \calV \rangle$ is finitely generated. Write $\{\calL_i\}_{0\leq i \leq n}$ for a set of generators of $\mathbb{X}$. We may then assume $\calV=\bigoplus_{0\leq i \leq n}\calL_i$. 
	
	\spa
	
	Let $\eta$ be a geometric point of $X\setminus D$. Write $I$ for the inertia subgroup of $\pie(X\setminus D,\eta)$ associated to $D$ and $\rho: \pie(X\setminus D,\eta)\to \GL(V)$ for the representation attached to $\calV$, where $V$ is a finite-dimensional $E$-linear vector space for some finite field extension $E/\Qp$. We want to construct $D'$ and $Z$ such that for every smooth curve $C\subseteq X$ not in $D$ and intersecting $D'$ in a point $x$ with $TC_x$ not in $Z_x$, the inertia subgroup of $C\setminus D$ at $x$, denoted by $I_{C,x}$, is mapped surjectively to $\rho(I)$. This would guarantee $R_\calV(Z)$.

	\spa
	
	Write $\Pi$ for the image of $\rho$, which by the assumption is the product of a finitely generated $\Zp$-module and a commutative finite group $\Pi_{\tors}$. Let $n$ be a positive integer such that $(\rho(I)\cap p^n \Pi)\subseteq p \cdot \rho(I)$ and let $\Gamma$ be the finite set of morphisms $\gamma:\Pi\to\Pi_{\tors}\times \Pi/p^n\Pi$ mapping $\rho(I)$ to a cyclic group $G_\gamma$ of prime order. We claim that to have $I_{C,x}\to I$ surjective it is enough to impose that $\gamma \circ \rho : I_{C,x}\to G_\gamma$ is surjective for every $\gamma\in \Gamma$. This follows from the combination of two facts. First, by Nakayama, the morphism $I_{C,x}\to \rho(I)$ is surjective if and only if $I_{C,x}\to \rho(I)/pm \cdot \rho(I)$ is surjective, where $m$ is defined as the product of the primes different from $p$ dividing the order of $\rho(I)_{\tors}$. Secondly, by Lemma \ref{group:l}, any quotient $\rho(I)\twoheadrightarrow G$ with $G$ cyclic lifts to a morphism $\Pi\to \Pi_{\tors}\times \Pi/p^n\Pi$ via some embedding $G\hookrightarrow \Pi_{\tors}\times  \Pi/p^n\Pi$.

	\spa

	For every $\gamma\in \Gamma$, let $f_\gamma:X_\gamma\to X$ be a connected finite étale cover such that the image of the restriction of $\gamma \circ \rho$ to the étale fundamental group of $X_\gamma\setminus D_\gamma$ is $G_\gamma$, where $D_\gamma:=f^{-1}_\gamma(D)$. Let $Y_\gamma\to X_\gamma\setminus D_\gamma$ be the associated connected $G_\gamma$-torsor. Write $\bar{Y}_\gamma$ for the normalisation of $X_\gamma$ in the field of functions of $Y_\gamma$. The morphism $\bar{Y}_\gamma \to X_\gamma$ is a finite cover ramified along $D_\gamma$. After shrinking $X_\gamma$ to an open $X'_\gamma$, we may assume that $\bar{Y}'_\gamma:=\bar{Y}_\gamma\times_{X_\gamma}X'_\gamma$ is regular and the image of $TY'_\gamma\to TX'_\gamma$ is a vector subbundle $\tilde{Z}_\gamma\subseteq TX'_\gamma$ of codimension $1$. We write $Z_\gamma\subseteq TX$ for the image of $\tilde{Z}_\gamma$ via the morphism $f_\gamma$ and $D'_\gamma$ for its projection to $X$. By construction, for every $x\in D'_\gamma$, we have that $Z_{\gamma,x}$ has codimension $1$ in $TX_x$. Write $D'$ for $\bigcap_{\gamma\in \Gamma} D'_\gamma$ and $Z$ for $\bigcup_{\gamma\in \Gamma} Z_{\gamma}|_{D'}$. For every $\gamma\in \Gamma$ the induced morphism $C\times_X\bar{Y}'_\gamma\to C\times_X{X}_\gamma$ is ramified at the points over $x$. Therefore, if $I_{C,x}$ is the inertia subgroup of the étale fundamental group of $C\setminus D$ at $x$, it is sent surjectively to $G_\gamma$ for every $\gamma\in \Gamma$. This yields the desired result. 
\end{proof}

\begin{lemm}\label{group:l}
Let $\Pi$ be a commutative profinite group containing a topologically finitely generated open pro-$p$-subgroup. Write $H\subseteq \Pi$ for a closed subgroup and let $n$ be a positive integer such that $(H\cap p^n \Pi)\subseteq p  H$. For every quotient $H\twoheadrightarrow G$ with $G$ cyclic there exists a morphism $\Pi\to\Pi_{\tors}\times\Pi/p^n \Pi$ and an embedding $G\hookrightarrow \Pi_{\tors}\times\Pi/p^n \Pi$ making the following diagram commute

 		\begin{equation*}
	\begin{tikzcd}
	\Pi\arrow[r] & \Pi_{\tors}\times \Pi/p^n\Pi\\
	H\arrow[r,two heads]\arrow[u,hook] & G\arrow[u,hook].
	\end{tikzcd}
\end{equation*}	

\end{lemm}
\begin{proof}
	If $\Pi$ is finite and $N$ is the kernel of $H\twoheadrightarrow G$, we can choose an embedding $\Pi/N\hookrightarrow \Pi$ which in turn induces an inclusion $G\hookrightarrow \Pi=\Pi_{\tors}\hookrightarrow  \Pi_{\tors} \times \Pi/p^n\Pi$ where the second arrow is the standard inclusion. These embeddings satisfy the desired property. If $\Pi$ is infinite we distinguish to cases. If the cardinality of $G$ is prime to $p$, then we deduce the statement from the result on $H_{\tors}\subseteq \Pi_\tors$ since $H\twoheadrightarrow G$ factors through any quotient $H\twoheadrightarrow H_{\tors}$. If the cardinality of $G$ is $p$, we note instead that by the assumption $(H\cap p^n \Pi)\subseteq p  H$, the quotient $H\twoheadrightarrow G$ factors through $H/(H\cap p^n \Pi)$. Therefore, we can deduce the result from the finite case looking at the groups $H/(H\cap p^n \Pi)\subseteq \Pi/ p^n\Pi$.
\end{proof}

We prove now the following consequence of Theorem \ref{main-curv:t}, which helps to describe the subgroup $X^*(G(\calMd))\subseteq X^*(G(\calM))$ \footnote{If $G$ is an algebraic group we denote by $X^*(G)$ its \textit{character  group}, i.e. the set of isomorphism classes of rank 1 representations endowed with tensor product.} in the docile situation. We suppose that $X$ admits a smooth compactification $X\subseteq Y\subseteq \PP^d_k$ and $Y\setminus X$ is a simple normal divisor of $D$. Write $X^*(G(\calM,V_\calM))^{\mathrm{ur}}$ for the group of characters of $G(\calM,V_\calM)$ corresponding to unramified rank 1 isocrystals with $\Qpu$-structure.

\begin{coro}\label{rk1-unr:c}
	If $X$ is a variety as in Proposition \ref{rk1-unr:p} and $(\calMd,\Phi_\calM^\ag)$ is an overconvergent $F^n$-isocrystal over $X$ docile along $D$, we have that $X^*(G(\calMd,V_\calM^\ag))=X^*(G(\calM,V_\calM))^{\mathrm{ur}}$, where $V_\calM^\ag$ and $V_\calM$ are the induced \DMs s.
\end{coro}
\begin{proof}To prove that $X^*(G(\calMd,V_\calM^\ag))\subseteq X^*(G(\calM,V_\calM))^{\mathrm{ur}}$ we combine two facts. First, the property of being docile is preserved under the operations of taking direct sum, tensor product, dual, and subquotients in $\oi(X)$ by \cite[Prop. 3.2.20]{Ked07}. Secondly, rank $1$ docile overconvergent isocrystals have 0 residues, thus they are unramified. This shows that rank 1 objects in $\langle \calMd,V_\calM^\ag\rangle$ are unramified.
	
	\spa
	
	For the other containment, by Proposition \ref{sq-DM:p} and Proposition \ref{obse-ffur-forg:p} every rank 1 isocrystal with $\Qpu$-structure in $\langle\calM, V_\calM\rangle$ comes from some $F^\infty$-isocrystal in $\langle\calM, \Phi^\infty_\calM\rangle$. The same result remains true for overconvergent isocrystals. Thanks to Proposition \ref{rk1-unr:p}, if $(\calL,\Phi_\calL^\infty)$ is an unramified (hence \dex) rank 1 $F^\infty$-isocrystal, then $(\calLd,\Phi_\calL^{\ag,\infty})\in \langle\calMd,\Phi^{\ag,\infty}_\calM\rangle.$ This concludes the proof. \end{proof}

 Finally, for Theorem \ref{tame-lef:t}, we need a Bertini's theorem, in the following form.

 \begin{theo}[Bertini's theorem]\label{Bertini:t}
 	Let $Y\subseteq \mathbb{P}^d_{\ka}$ be a smooth connected projective variety over $\ka$ of dimension $\geq2$, $D$ a divisor of $Y$, $D'$ an open subscheme of $D$, $Z$ a conic closed subset of $TY\times_Y D'$ which has codimension $1$ at each fibre, $U\subseteq Y$ a dense open and $\widetilde{U}\to U$ a connected finite étale cover. There exists a curve $C\subseteq Y$ over $\ka$ satisfying the following conditions.
 	\begin{itemize}
 		\item[{\normalfont(1)}] $C$ is a smooth scheme theoretic complete intersection of hyperplanes intersecting $D$ transversally.
 		
 		\item[{\normalfont(2)}]$C\times_U \widetilde{U}$ is connected and non-empty.
 		
 		\item[{\normalfont(3)}]$C$ intersects each irreducible component $D'_i$ of $D'$ and the image of $TC\times_YD_i'\to TY\times_Y D_i'$ is not entirely contained in $Z\times_{D'}D_i'$.
 		
 	\end{itemize}
 	
 \end{theo}
 \begin{proof}
 	By Bertini's theorem, in the form proven in \cite[Thm. 6.3]{Jou83}, the three conditions correspond to dense opens of the dual of $\mathbb{P}^d_{\ka}$ (see \cite[§1.8]{Del12} for more details). The result then follows from an induction on the dimension of $Y$. 
 \end{proof}
 
\subsubsection{Proof of Theorem \ref{tame-lef:t}.}
 	
 During the proof we will impose three conditions on $\overline{C}\subseteq Y$ numbered as the conditions in Theorem \ref{Bertini:t}. These conditions will ensure that if $C:=\overline{C}\setminus D$, the restriction functor $\langle \calMd \rangle \to \langle \calMd|_{C} \rangle$ is an equivalence. We first assume that $\overline{C}\subseteq Y$ is a smooth complete intersection of hyperplanes intersecting $D$ transversally (Condition (1)). By Theorem \ref{AE-lef:t}, the restriction functor $$\Isoc^\ag(X)_\mathrm{doc}\to \Isoc^\ag(C)_\mathrm{doc}$$ is fully faithful.
 	
 	\spa
 	
 	Let $U\subseteq X$ be a dense open where $\calM$ acquires the slope filtration and such that $D_S:=Y\setminus U$ is a divisor (containing $D$). Write $C_U$ for $C\cap U$ and $\calN$ for $\Gr_{S_\bullet}(\calM|_U)$. Choose a geometric point $\eta$ of $C_U$ and write $V$ for the \DM\ \Qpis\ of $(\calN,\Phi_\calN)$ at $\eta$. Let $\calV$ be the lisse $\Qpi$-sheaf associated to $(\calN,V)$ provided by Proposition \ref{etal-isoc:p} (which does not depend on $\eta$). By Lemma \ref{isoc-Lef:l}, there exists a connected finite étale cover $\widetilde{U}\to U$, with the property that if $\widetilde{U}\times_U C_U$ is connected and non-empty, then the functor $\langle\calN,V \rangle\to \langle \calN|_{C_U}, V\rangle$
 	is an equivalence of categories. Let us assume that $C_U$ satisfies this condition (Condition (2)).
 	
 	
 	\spa

 	 Since $\langle\calN,V \rangle\to \langle \calN|_{C_U},V \rangle$ is an equivalence, it follows that $X^*(G(\calM|_U,V))=X^*(G(\calM|_{C_U},V)).$ By Lemma \ref{ramif:l}, there exists a dense open $D_S'\subseteq D_S$ and a conic closed subscheme $Z\subseteq TX\times_X D_S'$ of codimension at least $1$ at every fibre which satisfies the property $\mathrm{R}_\calV(Z)$. Suppose that $\overline{C}$ satisfies the assumptions in $\mathrm{R}_\calV(Z)$ (Condition (3)), then $X^*(G(\calM|_U,V))^{\mathrm{ur}}=X^*(G(\calM|_{C_U},V))^{\mathrm{ur}}.$
 	
 	\spa
 	
 	We know that $\MS(X)$ is true for smooth curves thanks to Theorem \ref{main-curv:t}. Thus, by Theorem \ref{shrink-zar:t} and Corollary \ref{rk1-unr:c}, $$X^*(G(\calMd|_{C},V))=X^*(G(\calMd|_{C_U},V))= X^*(G(\calM|_{C_U},V))^{\mathrm{ur}}.$$ We showed that every rank $1$ overconvergent isocrystal with $\Qpi$-structure in $\langle \calMd|_{C}, V\rangle$ is the restriction of a rank 1 overconvergent isocrystal with $\Qpi$-structure over $Y$. This implies that the functor $\langle \calMd,V \rangle \to \langle \calMd|_C,V\rangle$ is observable. By Proposition \ref{comp:p}, we deduce that the restriction functor $\langle \calMd \rangle \to \langle \calMd|_C\rangle$ is an equivalence. It remains to show that a curve $\overline{C}$ with these three conditions exists, which is guaranteed by Theorem \ref{Bertini:t}.	\qed
 	
 	\spa
	



\subsubsection{Proof of Theorem \ref{main:t}}\label{pf-main-t:sss} By Theorem \ref{main-curv:t}, when $X$ is a smooth curve, $\MS(X)$ is true. To prove it in higher dimension, arguing as in Lemma \ref{shri-étal:l}, we may assume $k=\ka$. In addition, combining the reductions provided by Lemma \ref{shrink-zar:l} and Lemma \ref{shri-étal:l} and Kedlaya's semi-stable reduction theorem \cite[Thm. 2.4.4]{Ked11}, we may assume that $X$ admits a compactification as in Theorem \ref{tame-lef:t} and $(\calMd,\Phi_\calM^\ag)$ is an overconvergent $F$-isocrystal which is docile along $D$. Applying Theorem \ref{tame-lef:t}, there exists a curve $C$ in $X$ such that the \dex\ subobjects of $(\calM,\Phi_\calM)$ are the same as the ones of $(\calM|_C,\Phi_\calM|_C)$. Since we know that $\MS(\calM|_C)$ is true, we deduce that $\MS(\calM)$ is true as well. This ends the proof.\qed

\spa

As a consequence of Theorem \ref{main:t} we prove the parabolicity conjecture.
\begin{theo}\label{para-c:t}Let $X$ be a smooth geometrically connected variety over a perfect field $k$ endowed with the choice of a perfect point $\eta$. For every overconvergent $F^n$-isocrystal over $X$ with constant slopes, the subgroup $G(\calM,\eta)\subseteq G(\calMd,\eta)$ is the subgroup of $G(\calMd,\eta)$ stabilising the slope filtration of $\calM_\eta$. Moreover, if $\calMd$ is semi-simple, $G(\calM,\eta)$ is a parabolic subgroup of $G(\calMd,\eta)$.
\end{theo}
\begin{proof}By \cite[(2.1.10)]{Cre92a}, we may assume $k=\ka$ and that $\eta$ is a geometric point. By Theorem \ref{main:t} and Proposition \ref{para:p}, we deduce that $G(\calM,\Phi_\calM^\infty,\eta)\subseteq G(\calMd,\Phi_\calM^{\ag,\infty},\eta)$ is the stabiliser of the slope filtration of $\calM_\eta$. In addition, thanks to Corollary \ref{inters:c}, we have that $G(\calM,\eta)=G(\calMd,\eta)\cap G(\calM,\Phi^\infty_\calM,\eta)\otimes_{\Qpu} K$. This proves the first part of the statement. If $\calMd$ is semi-simple, up to replacing $(\calMd,\Phi_\calM^\ag)$ with its semi-simplification with respect to a Jordan--H\"older filtration, we may assume that $(\calMd,\Phi_\calM^\ag)$ is semi-simple. By Proposition \ref{para:p}, the subgroup $G(\calM,\Phi_\calM^\infty,\eta)\subseteq G(\calMd,\Phi_\calM^{\ag,\infty},\eta)$ is then parabolic. Since $G(\calM,\eta)\subseteq G(\calMd,\eta)$ is a normal subgroup, this implies that $G(\calM,\eta)=G(\calMd,\eta)\cap G(\calM,\Phi^\infty_\calM,\eta)\otimes_{\Qpu} K$ is a parabolic subgroup of $G(\calMd,\eta)$.
\end{proof}

 \section{Applications}\label{apps:s}
 \subsection{Monodromy over finite fields}\label{fini-fiel:ss}

 Let $X$ be a smooth connected variety over $\F_{p^n}$ with a rational point $x$. In this case, the fibre functor $\omega_x:\Isoc(X)\to \VVec_{\Qpn}$ induces a $\Qpn$-linear fibre functor for the $\Qpn$-linear Tannakian category $\omega_x:\Fnisoc(X)\to\VVec_{\Qpn}$.
 \begin{defi}
 	 For $(\calM,\Phi_\calM)\in \Fnisoc(X)$ (resp. $(\calMd,\Phi^\ag_\calM)\in \Fnisoc(X)$), we write \break$G(\calM,\Phi_\calM,x)$ (resp. $G(\calMd,\Phi^\ag_\calM,x)$) for the algebraic group of automorphisms of the restriction of $\omega_x$ to $\langle \calM,\Phi_\calM \rangle$ (resp. $\langle \calMd,\Phi^\ag_\calM \rangle$). The group $G(\calMd,\Phi^\ag_\calM,x)$ coincides with the arithmetic monodromy group of $(\calMd,\Phi_\calM^\ag)$ defined in \cite[Def. 3.2.4]{Dad}.
 \end{defi}
  By \cite[Prop. 2.2.4]{AD18}, for every overconvergent $F^n$-isocrystal $(\calM^\dagger,\Phi_\calM^\ag)$ over $X$, we have, as in Proposition \ref{fund-exac-seq:p}, the following commutative diagram of $\Q_{p^n}$-linear algebraic groups
 \begin{equation}\label{fund-diag:e}
 	\begin{tikzcd}
 		1\arrow{r} & G(\calM,x)\arrow{r}\arrow[hook,d] & G(\calM,\Phi_\calM,x)\arrow{r}\arrow[hook,d] & G(\calM,\Phi_\calM,x)^{\cst}\arrow{r}\arrow[d] &1\\
 		1\arrow{r} & G(\calM^\ag,x)\arrow{r} & G(\calM^\ag, \Phi_{\calM}^\ag, x)\arrow{r} & G(\calM^\ag,\Phi_{\calM}^\ag,x)^{\cst}\arrow{r} & 1
 	\end{tikzcd}
 \end{equation}
where the rows are exact. The groups on the right are the Tannaka groups of the category of constant objects in $\langle \calM, \Phi_\calM \rangle$ and $\langle \calMd,\Phi^\ag_\calM \rangle$. We write $G$ for $G(\calM^\ag, \Phi_{\calM}^\ag, x)$ and $H$ for $G(\calM,\Phi_\calM,x)$. As in Definition \ref{coch:d}, the slope filtration of $\calM_x$ defines a quasi-cocharacter $\lambda:\Gm^{1/\infty}\to G(\calMd,\Phi^\ag_\calM,x)$. Write $P_G(\lambda)\subseteq G$ for the stabiliser of the slope filtration.

 \begin{theo}\label{main-fini:t}
If $(\calM,\Phi_\calM)$ has constant slopes, $H=P_G(\lambda)$. Moreover, if $(\calMd,\Phi_\calM^\ag)$ is semi-simple, $H$ is a parabolic subgroup of $G$.
 \end{theo}
\begin{proof}
This follows from Theorem \ref{main:t} by arguing as in Proposition \ref{para:p}.
\end{proof}

\begin{prop}
	For an overconvergent $F^n$-isocrystal $(\calMd,\Phi_\calM^\ag)$ (possibly with non-constant slopes) $\pi_0(H)=\pi_0(G)$. In addition, the morphism $G(\calM,\Phi_{\calM},x)^{\cst}\to G(\calM^\ag,\Phi_{\calM}^\ag,x)^{\cst}$ is an isomorphism.
\end{prop}

\begin{proof}
This follows from Theorem \ref{main-fini:t} by arguing as in Proposition \ref{pi0:p} and Proposition \ref{rk1-unr:p}.
\end{proof}

Suppose that $(\calM,\Phi_\calM)$ has constant slopes. If we write $(\calN,\Phi_\calN)$ for $\Gr_{S_\bullet}(\calM,\Phi_\calM)$, there is a functor $\langle \calM, \Phi_\calM \rangle\to \langle\calN,\Phi_\calN\rangle$ sending $(\calM',\Phi_{\calM'})$ to $\Gr_{S_\bullet}(\calM',\Phi_{\calM'})$. This induces the following commutative diagram with exact rows 
 
  \begin{equation}\label{graded:e}
 	\begin{tikzcd}
 		1\arrow{r} & G(\calN,x)\arrow{r}\arrow[hook,d] & G(\calN,\Phi_\calN,x)\arrow{r}\arrow[hook,d] & G(\calN,\Phi_\calN,x)^{\cst}\arrow{r}\arrow[d, "="] &1\\
 		1\arrow{r} & G(\calM,x)\arrow{r} & G(\calM, \Phi_{\calM}, x)\arrow{r} & G(\calM,\Phi_{\calM},x)^{\cst}\arrow{r} & 1.
 	\end{tikzcd}
 \end{equation}
 
 Even in this case, the natural morphism $G(\calN,\Phi_{\calN},x)^{\cst}\to G(\calM,\Phi_{\calM},x)^{\cst}$ is an isomorphism, since the inclusion $\langle \calN,\Phi_\calN \rangle\hookrightarrow \langle \calM,\Phi_\calM \rangle$ provides an inverse map.
 
 \begin{prop}\label{cent:p}
 The subgroup $G(\calN,\Phi_\calN,x)\subseteq P_G(\lambda)$ is equal to $Z_G(\lambda)$, the centraliser of the image of $\lambda$.
 \end{prop}
\begin{proof}
The proof is similar to Proposition \ref{para:p}. By construction, the subgroup $G(\calN,\Phi_\calN,x)$ is in $Z_G(\lambda)$. On the other hand, by Chevalley's theorem, there exists an $F^n$-isocrystal $(\calM',\Phi_{\calM'})\in\langle \calM, \Phi_\calM\rangle$ and a subobject $(\calL,\Phi_\calL)\subseteq \Gr_{S_\bullet}(\calM',\Phi_{\calM'})$ of rank $1$ such that $G(\calN,\Phi_\calN,x)$ is the stabiliser of $L:=\omega_x(\calL)\subseteq \omega_x(\calM')=:V$. Since $(\calL,\Phi_\calL)$ has rank $1$, it is contained as $F^n$-isocrystal in the image of the quotient $\pi:S_i(\calM')\to S_i(\calM')/S_{i-1}(\calM')$ for some $i$. Write $(\widetilde{\calL},\Phi_{\widetilde{\calL}})\subseteq (\calM',\Phi_{\calM'})$ for $\pi^{-1}(\calL,\Phi_\calL)$ and $\widetilde{L}\subseteq V$ for its fibre at $x$. If $s$ is the slope of $(\calL,\Phi_\calL)$ and $V^s\subseteq V$ is the subspace of slope $s$, then $L=\widetilde{L}\cap V^s$. We deduce that $Z_G(\lambda)$ stabilises $L$, which gives the desired result.
\end{proof}

\begin{coro}
The algebraic group $G(\calN,\Phi_\calN,x)^\circ$ contains a Cartan subgroup of $G(\calMd,\Phi^\ag_\calM,x)^\circ$.
\end{coro}
\begin{proof}
If $T$ is a maximal torus of $G(\calMd,\Phi^\ag_\calM,x)^\circ$ containing the image of $\lambda$, the centraliser of $T$ in $G(\calMd,\Phi^\ag_\calM,x)$ is contained in $Z_G(\lambda)=G(\calN,\Phi_\calN,x)$. This yields the desired result.
\end{proof}
 
 \begin{theo}\label{semi-simp:t}
 Let $X$ be a smooth variety over a finite field $\F_{p^n}$ and $f:A\to X$ an abelian scheme with constant slopes. If $(\calM,\Phi_\calM)$ is the $F$-isocrystal $R^1f_{\crys*}\calO_{A,\crys}$, the induced $F$-isocrystal $(\calN,\Phi_\calN):=\Gr_{S_\bullet}(\calM,\Phi_\calM)$ is semi-simple. In particular, $R^1f_{\et*}\underline{\Qp}$ is a semi-simple lisse $\Qp$-sheaf over $X$.
 \end{theo}

\begin{proof}
By étale descent we may assume that $X$ is connected and admits a rational point $x$. We may also replace $\Phi_{\calM}$ with its $n$-th power. By \cite[Thm. 7]{Ete02}, the $F^n$-isocrystal $(\calM,\Phi_\calM)$ is \dex\ and, by \cite[Cor. 3.5.2.(ii)]{Dad}, the monodromy group $G(\calMd,x)$ is a reductive group (note that ($\calMd,\Phi^\ag_\calM$) is pure by the Riemann Hypothesis for abelian varieties). On the other hand, since the action of the $p^n$-th power Frobenius on the crystalline cohomology groups of $A_x$ is semi-simple (this Frobenius is in the centre of $\End(A_x)$), we get that $G(\calMd,\Phi^\ag_\calM,x)^{\cst}$ is a reductive group. We deduce by (\ref{fund-diag:e}) that $G(\calMd,\Phi^\ag_\calM,x)$ is reductive (see also \cite[Thm. 1.2]{Pal15} for a different proof over curves). By Proposition \ref{cent:p}, the group $G(\calN,\Phi_\calN,x)$ is the centraliser of $\lambda$ in $G(\calM,\Phi_\calM,x)$, thus by \cite[Cor. 11.12]{Bor91} it is reductive. This shows that $(\calN,\Phi_\calN)$ is semi-simple.

\spa

 To prove that $R^1f_{\et*}\underline{\Qp}$ is semi-simple it is enough to prove that the associated unit-root $F$-isocrystal $(\calV,\Phi_\calV)$ is semi-simple. This $F$-isocrystal coincides with the crystalline Dieudonné module of the $p$-divisible group $A[p^{\infty}]^{\et}$. By  \cite[Thm. 2.5.6.(ii)]{BBM82}, the $F$-isocrystal $(\calM,\Phi_\calM)$ is instead the crystalline Dieudonné module of $A[p^{\infty}]$, thus $(\calV,\Phi_\calV)$ is the minimal slope sub-$F$-isocrystal of $(\calM,\Phi_\calM)$. This concludes the proof.
\end{proof}

 \subsection{Separable points of abelian varieties}\label{sep-pts:ss}
 Let $E/\Fp$ be a finitely generated field extension and let $A$ be an abelian variety over $E$. We use next consequence of Theorem \ref{main:t} to prove a result on torsion $E^{\sep}$-points of $A$.
 \begin{coro}\label{max-slope-sub:c}
 	If $(\calM,\Phi_\calM)$ is a \dex\ $F^n$-isocrystal over a smooth variety $X$ over $k$, every isoclinic subobject of maximal slope is \dex.
 \end{coro}
\begin{proof}
After shrinking $X$ we may assume that the slopes of $(\calM,\Phi_\calM)$ are constant. By Theorem \ref{main:t}, a subobject $(\calN,\Phi_\calN)\subseteq (\calM,\Phi_\calM)$ of maximal slope is equal to its \dhl. This yields the desired result.
\end{proof}
 \begin{theo}\label{abelian:t}
 	Let $A$ be an abelian variety over $E$. The group $A(E^{\mathrm{sep}})[p^\infty]$ is finite in the following two cases. 
 	\begin{itemize}
 		\item[(i)] If $\End(A)\otimes_\Z \Qp$ is a division algebra.
 		\item[(ii)] If $\End(A)\otimes_\Z \Q$ has no factor of Albert-type IV.
 	\end{itemize}
 \end{theo}
 \begin{proof}
 	
 	We suppose by contradiction that the group $A(E^{\mathrm{sep}})[p^\infty]$ is infinite and we want to show that both (i) and (ii) cannot occur. Let us make first some preliminary observations. Since $A(E^{\mathrm{sep}})[p^\infty]$ is infinite, the \bt group $A_{E^{\mathrm{sep}}}[p^\infty]$ over $E^{\sep}$ admits $\Qp/\Zp$ as a subgroup. Let $\widetilde{H}$ be the maximal constant \bt subgroup of $A_{E^{\mathrm{sep}}}[p^\infty]$. By Galois descent, $\widetilde{H}$ descends to some étale \bt group $H$ over $E$. Let $X$ be a smooth connected variety over $\Fp$ with function field $E$. After shrinking $X$ we may assume that $A$ over $E$ admits a smooth model $\calA\to X$ with constant Newton polygon. Let $(\calM,\Phi_\calM)$ be the crystalline Dieudonné module of $\calA[p^\infty]$ and $(\calN,\Phi_\calN)$ the crystalline Dieudonné module of the model of $H$ in $\calA[p^\infty]$. Since $H$ is étale, the $F$-isocrystal $(\calN,\Phi_\calN)$ is a unit-root quotient of $(\calM,\Phi_\calM)$. At the same time, we know that $(\calM,\Phi_\calM)$ is \dex, thus by Corollary \ref{max-slope-sub:c} the isocrystal $(\calN,\Phi_\calN)$ is \dex\ as well. As usual, write $(\calN^\ag,\Phi_\calN^\ag)$ for its $\dag$-extension. 
 	
 	\spa
 	Let us now prove that $\End(A)\otimes_\Z \Qp$ contains a non-trivial idempotent, showing thereby that (i) is not possible.  Arguing as in the proof of Theorem \ref{semi-simp:t}, we know that the monodromy group of $(\calM^\ag,\Phi^\ag_\calM)$ is reductive, which implies that $(\calM^\ag,\Phi^\ag_\calM)$ is semi-simple. Thus $(\calN^\ag,\Phi_\calN^\ag)$ is actually a proper direct summand of $(\calM^\ag,\Phi_\calM^\ag)$. This implies that $\End(\calM^\ag,\Phi_\calM^\ag)$ contains some non-trivial idempotent. Thanks to \cite[Thm. 2.6]{deJ98}, we deduce that the same is true for $\End(A)\otimes_\Z \Qp$. 
 	
 	\spa
 	
 	It remains to prove that $A$ admits a factor of Albert-type IV. For this we may assume $A$ geometrically simple. Thanks to \cite[Lem. 10.5]{CT20}, it is enough to show that the centre of the neutral component $G(\calMd,\Phi^\ag_\calM)^\circ$ is of dimension at least $2$. First, by \cite[Prop. 3.3.4]{Dad}, up to replacing $E$ with a finite field extension (and $X$ with some finite étale cover) we may assume that $G(\calMd,\Phi^\ag_\calM)$ is connected. Then, thanks to \cite[Cor. 3.4.8]{Dad}, we know that the rank of the group of \textit{twist classes} $\mathbb{X}(\calMd)\subseteq \Qpbar^\times/\mu_\infty(\Qpbar)$ (cf. Def. 3.4.3 in [\textit{ibid.}]) is equal to the dimension of the centre of $G(\calMd,\Phi^\ag_\calM)$. In addition, since $(\calM^\ag,\Phi_\calM^\ag)$ is the overconvergent $F$-isocrystals associated to an abelian scheme, its determinant has twist class $[p^g]\in \mathbb{X}(\calMd)$, where $g$ is the dimension of $A$. Write $[a]\in \mathbb{X}(\calMd)$ for a twist class of $(\calNd,\Phi_\calN^\ag)$. If we prove that $[a]$ is not a $\Q$-multiple of $[p^g]$, we deduce that $\mathbb{X}(\calMd)$ is of rank at least $2$, so that the centre of $G(\calMd,\Phi^\ag_\calM)$ is of dimension at least $2$, as we want. Suppose by contradiction that $\tfrac{a^{n}}{p^{mg}}$ is a root of unity for some $(m,n) \in \Z\times \Z_{>0}$. Since $a$ is a $p$-adic unit, we have that $m=0$. This implies that $a$ is a root of unity itself. On the other hand, by the Riemann Hypothesis, $a$ is a Weil number of weight $1$. This leads to a contradiction.
 \end{proof}
 
 \begin{rema}
 	In the proof of Theorem \ref{abelian:t} one could deduce that the unit-root $F$-isocrystal $(\calN,\Phi_\calN)$ is \dex\ by replacing Corollary \ref{max-slope-sub:c} with \cite[Thm. 1.2.(a)]{Ros17}.  It is also worth mentioning that Helm \cite{Hel20} recently found an example of an ordinary abelian variety over a certain finitely generated field $E$ with no isotrivial factors and such that $A(E^{\mathrm{sep}})[p^\infty]$ is infinite. 
 \end{rema}
  \subsection{Kedlaya's conjecture}\label{k-c:ss} Thanks to Theorem \ref{main:t}, we are also able to prove a conjecture proposed by Kedlaya.

 \begin{coro}[Kedlaya's conjecture]\label{k-c:c}
 	Let $X$ be a smooth connected variety over a perfect field $k$ and let $(\calM_1^\ag,\Phi_{\calM_1}^\ag)$ and $(\calM_2^\ag,\Phi_{\calM_2}^\ag)$ be two irreducible overconvergent $F^n$-isocrystals over $X$ with constant slopes. If there exists an isomorphism $\epsilon: (S_1(\calM_1),\Phi_{\calM_1}|_{S_1(\calM_1)})\iso (S_1(\calM_2),\Phi_{\calM_2}|_{S_1(\calM_2)})$, then $(\calM_1^\ag,\Phi_{\calM_1}^\ag)\simeq(\calM_2^\ag,\Phi_{\calM_2}^\ag)$.
 \end{coro}
 \begin{proof}
 	Write $(\calN,\Phi_\calN)$ for $(S_1(\calM_1),\Phi_{\calM_1}|_{S_1(\calM_1)})$ and $\iota:(\calN,\Phi_\calN)\hookrightarrow (\calM_1,\Phi_{\calM_1})$ for the tautological inclusion. If $(\calM,\Phi_\calM)$ is the direct sum $(\calM_1,\Phi_{\calM_1})\oplus (\calM_2,\Phi_{\calM_2})$, we have an inclusion $(\iota,\epsilon):(\calN,\Phi_\calN)\hookrightarrow (\calM,\Phi_\calM)$. By Theorem \ref{main:t}, the isocrystal $\caloN$ is a proper subobject of $\calM$ since $S_1(\caloN)=\calN \neq \calN \oplus \calN = S_1(\calM)$. In turn, by the assumption, this implies that $(\caloN^\ag,\Phi_\caloN^\ag)$ is irreducible. On the other hand, by construction, $(\caloNd,\Phi_\caloN^\ag)$ admits non-zero morphisms to $(\calM_1^\ag,\Phi_{\calM_1}^\ag)$ and $(\calM_2^\ag,\Phi_{\calM_2}^\ag)$. Therefore we get $$(\calM_1^\ag,\Phi_{\calM_1}^\ag)\simeq (\caloNd,\Phi_\caloN^\ag) \simeq (\calM_2^\ag,\Phi_{\calM_2}^\ag),$$ as we wanted.
 \end{proof}
 In addition, we can prove the next related result.

 \begin{coro}\label{irr:c} Let $X$ be a smooth connected variety over a perfect field $k$. If $(\calMd,\Phi_\calM^\ag)$ is an irreducible overconvergent $F^n$-isocrystal with constant slopes, then $(S_1(\calM),\Phi_\calM|_{S_1(\calM)})$ is irreducible.
 \end{coro}
 
 \begin{proof}
 	Let $(\calN,\Phi_\calN)\subseteq (S_1(\calM),\Phi_\calM|_{S_1(\calM)})$ be an irreducible sub-$F^n$-isocrystal. Since $(\calMd,\Phi_\calM^\ag)$ is irreducible, the $\dag$-hull of $(\calN,\Phi_\calN)$ is equal to $(\calMd,\Phi_\calM^\ag)$. By Theorem \ref{main:t} we deduce that $\calN=S_1(\caloN)=S_1(\calM)$. This yields the desired result.
 \end{proof}

Combining Kedlaya's conjecture over finite fields, already proved in \cite{Tsu19}, and Abe's Langlands correspondence we deduce the following $p$-adic refinement of the strong multiplicity one theorem for cuspidal automorphic representations.
 
 \begin{theo}\label{mult-one:t}
	Let $X$ be a smooth connected curve over a finite field, let $\mathbb{A}$ be its adele ring and let $r$ be a positive integer. The isomorphism class of a $\Qpbar$-linear cuspidal automorphic representation $\pi$ of $\GL_r(\mathbb{A})$ is determined by the datum of the Hecke eigenvalues of minimal slope at all but finitely many closed points of $X$.
 \end{theo}
 \begin{proof}
 	By \cite{Abe}, for any such $\pi$ there exists an irreducible $\Qpbar$-linear overconvergent $F$-isocrystal $(\calMd,\Phi_\calM^\ag)$ defined over a certain dense open of $X$ which corresponds to $\pi$ in the sense of Langlands. After shrinking $X$ we may assume that $(\calMd,\Phi_\calM^\ag)$ has constant slopes. By Corollary \ref{k-c:c}, the $F$-isocrystal $(S_1(\calM),\Phi_\calM|_{S_1(\calM)})$ determines the isomorphism class of $\calM$. On the other hand, by \cite[Thm. 2.1]{Cre87}, the isoclinic $F$-isocrystal $S_1(\calM)$ after twist is induced by a $\Qpbar$-linear continuous representation $\rho$ of the étale fundamental group of $X$. Moreover, by Corollary \ref{irr:c}, $\rho$ is irreducible. Thanks to Chebotarev's density theorem for the étale fundamental group of $X$ proved in \cite[Thm. 7]{Ser66}, the isomorphism class of $\rho$ is determined by the Frobenius eigenvalues at all but finitely many points. By construction, these eigenvalues are the same as the Frobenius eigenvalues of $(\calMd,\Phi_\calM^\ag)$ of minimal slope. Since $\pi$ and $(\calMd,\Phi_\calM^\ag)$ correspond in the sense of Langlands, we obtain the desired result.
 \end{proof}

 \subsection{\texorpdfstring{\PBS}{} filtration and \texorpdfstring{$\dag$}{}-compactifications}\label{PBS:ss}
 
 In this last section we extend \cite[Thm. 3.27]{Tsu19} to general smooth varieties and we introduce the notion of \textit{$\dag$-compactification} of an $F^n$-isocrystals, which is used to state Corollary \ref{gen-Ked-con:c}, a stronger form of Corollary \ref{k-c:c}.

 \begin{defi}
 	We say that a $\dag$-extendable $F^n$-isocrystal over $X$ is \textit{pure in bounded subobjects} or simply \PBS\ if for every connected open $U\subseteq X$, the isoclinic subobjects of the restriction to $U$ have minimal generic slope\footnote{This is the dual of Tsuzuki's notion of $\mathrm{PBQ}$ overconvergent $F^n$-isocrystals.}.
 \end{defi}

 \begin{coro}\label{PBS:c}
 	Let $\calM$ be a $\dag$-extendable $F^n$-isocrystal over a smooth connected variety over a perfect field $k$. There exists a unique filtration $0=P_{0}(\calM)\subsetneq P_1(\calM)\subsetneq \dots \subsetneq P_r(\calM)=\calM$ of $\dag$-extendable $F^n$-isocrystals such that each quotient $P_i(\calM)/P_{i-1}(\calM)$ is \PBS\ with minimal generic slope $t_i\in \Q$ and $t_1>t_2>\dots >t_r$. 
 \end{coro}
 
 \begin{proof}
 	After shrinking the variety in order to have constant slopes, we construct $P_{r-1}(\calM)$ as the $\dag$-hull of the sum of all the isoclinic subobjects of $\calM$ which do not have minimal generic slope. By Theorem \ref{main:t}, we deduce that $S_1(\calM)\cap P_{r-1}(\calM)=0$. Therefore, if $\calM\neq 0$ then the quotient $\calM/P_{r-1}(\calM)$ is not $0$. On the other hand, $\calM/P_{r-1}(\calM)\neq0$ is \PBS\ by construction. The result then follows by an induction on the rank of $\calM$.
 \end{proof}

 \begin{defi} If $(\calN,\Phi_\calN)$ is an $F^n$-isocrystal, we say that a \dex\ $F^n$-isocrystal $(\calM,\Phi_\calM)$ endowed with an inclusion $\iota_\calM:(\calN,\Phi_\calN)\hookrightarrow (\calM,\Phi_\calM)$ such that $\caloN=\calM$ is a \textit{\dc} of $(\calN,\Phi_\calN)$. A morphism $\psi: (\calM_1,\Phi_{\calM_1})\to(\calM_2,\Phi_{\calM_2})$ of $\dag$-compactifications is a morphism of $F^n$-isocrystals such that $\psi\circ\iota_{\calM_1}=\iota_{\calM_2}$. We say that a weakly final object of the category of $\dag$-compactifications of $\calN$ is a \textit{minimal $\dag$-compactification} of $\calN$.
 \end{defi}

 \begin{coro}\label{gen-Ked-con:c}
 	If $(\calN,\Phi_\calN)$ is isoclinic and it can be embedded into a \dex\ $F^n$-isocrystal, then it admits a minimal $\dag$-compactification.
 \end{coro}
 \begin{proof}
 	Since the category of $F^n$-isocrystals is noetherian, it is enough to prove that for every pair $(\calM_1,\Phi_{\calM_1}), (\calM_2,\Phi_{\calM_2})$ of \dcs\ of $(\calN,\Phi_\calN)$, we can find isomorphic \dcs\ $(\calM_1',\Phi_{\calM_1'})$ and $(\calM_2',\Phi_{\calM_2'})$ with (surjective) morphisms $(\calM_i,\Phi_{\calM_i})\twoheadrightarrow (\calM_i',\Phi_{\calM_i'})$. Write $(\calM,\Phi_\calM)$ for $(\calM_1,\Phi_{\calM_1})\oplus (\calM_2,\Phi_{\calM_2})$ and endow $(\calN,\Phi_\calN)$ with the “diagonal” inclusion in $(\calM,\Phi_\calM)$ induced by $\iota_{\calM_1}$ and $\iota_{\calM_2}$. Write $(\caloN,\Phi_{\caloN})$ for the $\dag$-hull of $(\calN,\Phi_\calN)$ in $(\calM,\Phi_\calM)$. We consider the isocrystals $\calM_i':=\calM_i/(\calM_i\cap \caloN)$ endowed with their natural $F^n$-structures $\Phi_{\calM_i'}$ induced by $\Phi_{\calM_i}$.
 	
 	\spa
 	We claim that the induced morphisms $\iota_{\calM_i'}:\calN\to\calM_i'$ are injective. To prove this, we may shrink $X$ in order to acquire the slope filtration. Since $(\calN,\Phi_\calN)$ is isoclinic, by Theorem \ref{main:t} the subobject $\caloN\cap \iota_{\calM_i}(\calN)\subseteq \caloN$ lies in $S_1(\caloN)=\calN$. This implies that $\calM_i\cap \caloN\cap \iota_{\calM_i}(\calN)=  \calM_i\cap\calN=0$, as we wanted. 
 To end the proof we notice that the two projections $\mathrm{pr}_i:\calM\to \calM_i$ send $\caloN$ to $\calM_i$ since $(\calM_i,\Phi_{\calM_i})$ are $\dag$-hulls of $(\calN,\Phi_\calN)$. This implies that the morphisms $\calM_i\to \calM/\caloN$ are surjective, thus $$\calM_1':=\calM_1/(\calM_1\cap \caloN)\simeq \calM/\caloN \simeq \calM_2/(\calM_2\cap \caloN)=:\calM_2'.$$ Therefore, $(\calM_1',\Phi_{\calM_1'})$ and $(\calM_2',\Phi_{\calM_2'})$ endowed with the morphisms $\iota_{\calM_1'}$ and $\iota_{\calM_2'}$, satisfy all the desired properties.

 \end{proof}


	\bibliographystyle{ams-alpha}

\end{document}